\documentclass{article}
\usepackage{graphicx}
\usepackage{epstopdf}
\usepackage{amsmath}
\usepackage{amssymb}
\usepackage{bm}
\usepackage{bbm}
\usepackage{enumerate}
\usepackage[T1]{fontenc}
\usepackage[latin1]{inputenc}
\usepackage{hyperref}
\usepackage{color}
\usepackage{curves}
\usepackage{tikz}
\usepackage[center]{caption}
\usepackage{ntheorem}
\usepackage{float}
\usepackage{mathrsfs}

\title{Further results on consensus formation\\
       in the Deffuant model}
\author{Olle Häggström \thanks{Research supported by grants
 from the Swedish Research Council and from the Knut and Alice
 Wallenberg Foundation}
\qquad Timo Hirscher
     \thanks{Research supported by a grant from the Swedish Research Council}
\\\normalsize Chalmers University of Technology}
               
\theoremstyle{break}
\newtheorem{theorem}{Theorem}[section]

\newtheorem{lemma}{Lemma}[section]
\newtheorem{proposition}{Proposition}[section]
\theorembodyfont{\upshape}
\newtheorem{definition}{Definition}
\newtheorem{remark}{Remark}[section]
\newtheorem*{example}{Example}
\newtheorem*{examples}{Examples}
\makeatletter
\let\c@proposition\c@theorem
\let\c@lemma\c@theorem
\let\c@corollary\c@theorem
\makeatother

\newenvironment{proof}{\noindent{\sc Proof:}}{\vspace{-0.5cm}~\hfill $\square$\vspace{0.5cm}}
\newenvironment{nproof}[1]{\noindent{\sc Proof #1:}}{\vspace{-1em}~\hfill $\square$\vspace{2em}}

\newcommand\N{\mathbb{N}}
\newcommand\R{\mathbb{R}}
\newcommand\Z{\mathbb{Z}}
\newcommand\E{\mathbb{E}\,}
\newcommand\Prob{\mathbb{P}}
\renewcommand\epsilon{\varepsilon}
\renewcommand\phi{\varphi}
\DeclareMathOperator\supp{supp}
\newcommand\En{\mathcal{E}}

\def\law{\buildrel d \over\rightarrow}

\definecolor{darkblue}{rgb}{0,0,.5}
\hypersetup{colorlinks=true, breaklinks=true, linkcolor=blue, 
citecolor=darkblue, menucolor=blue, urlcolor=blue}

\begin{document}
\newpage
\maketitle
\begin{abstract}
    The so-called Deffuant model describes a pattern for social interaction,
    in which two neighboring individuals randomly meet and share their opinions
    on a certain topic, if their discrepancy is not beyond a given threshold $\theta$.
    The major focus of the analyses, both theoretical and based on simulations, lies on
    whether these single interactions lead to a global consensus in the long run or not.
    First, we generalize a result of Lanchier for the Deffuant model on $\Z$, determining
    the critical value for $\theta$ at which a phase transition of the long term
    behavior takes place, to other distributions of the initial opinions than i.i.d.\
    uniform on $[0,1]$. Then we shed light on the situations where the underlying
    line graph $\Z$ is replaced by higher-dimensional lattices $\Z^d,\ d\geq2$, or the
    infinite cluster of supercritical i.i.d.\ bond percolation on these lattices.
\end{abstract}



\section{Introduction}

Let $G=(V,E)$ be a simple graph, i.e.\ having undirected edges and neither loops nor multiple edges. The considered
graph may either be finite or infinite with bounded maximal degree. Furthermore, without loss
of generality we can assume $G$ to be connected, since in what follows one could consider the connected
components seperately otherwise. Every vertex is understood to represent an individual and will at each time
$t\geq0$ be assigned a value representing its opinion.
All the edges in $E$ are connections between individuals allowing for mutual influence.
There are a number of models for what is called {\itshape opinion dynamics}, which
are qualitatively different but share similar ideas, see \cite{Survey} for an extensive
survey.\vspace{0.8em}

\noindent The {\em Deffuant model} (introduced by Deffuant et al.\ \cite{Model})
featuring two model parameters $\mu\in(0,\tfrac 12]$ and
$\theta\in(0,\infty)$ is defined as follows. At time $t=0$, the vertices are assigned i.i.d.\ initial
opinions, in the standard case uniformly distributed on the interval $[0,1]$.
In addition, serving as a regime for the random encounters, every edge $e\in E$ is
assigned a unit rate Poisson process.
The latter are independent of each other and the initial distribution of opinion values.
Denote the opinion value at $v\in V$ at time $t$ by $\eta_t(v)$, which remains
unchanged until at some time $t$ a Poisson event occurs at an edges incident to $v$,
say $e=\langle u,v\rangle$. The opinion values of $u$ and $v$ just before this happens may be
$\eta_{t-}(u)=\lim_{s\uparrow t}\eta_s(u)=:a$ and $\eta_{t-}(v)=\lim_{s\uparrow t}\eta_s(v)=:b$ respectively.

If these values are within the confidence bound $\theta$, they come symmetrically closer to each
other, if not they stay unchanged, i.e.
 \begin{equation*}
         \eta_t(u) = \left\{ \begin{array}{ll}
                          a+\mu(b-a) & \mbox{if $|a-b|\leq\theta$,} \\
                          a & \mbox{otherwise}
                         \end{array} \right.
 \end{equation*}                     
and similarly \vspace{-0.75cm}\begin{align}\label{dynamics}\end{align}\vspace{-0.75cm}

 \begin{equation*}
         \eta_t(v) = \left\{ \begin{array}{ll}
                          b+\mu(a-b) & \mbox{if $|a-b|\leq\theta$,} \\
                          b & \mbox{otherwise.}
                         \end{array} \right.
 \end{equation*}
Observe that $\mu$ is modelling the willingness of the individuals to step towards
other opinions encountered that fall within their interval of tolerance, shaped by $\theta$. In other
words, a value of $\mu$ close to $0$ represents a strong reluctance to change one's mind.
For the process to be well-defined, on the one hand one has to make sure that neither two Poisson
events occur simultaneously nor that there is a limit point in time for the events occuring on edges
incident to one fixed vertex. But since the maximal degree is bounded and we assume the vertex
set to be countable, this is almost surely the case. On the other hand, there is a more subtle
issue in how the simple interactions shape transitions of the whole system on an infinite graph
-- is it well-defined there as well?
For infinite graphs with bounded degree, this problem is settled by standard techniques
in the theory of interacting particle systems, see Thm.\ 3.9 on p.\ 27 in \cite{Liggett}.
\vspace*{1em}

\noindent The most natural question to ask seems to be, if the individual opinions will converge to
a common consensus in the long run or if they are going to be split up into groups of
individuals holding different opinions. In this regard let us define the following types of scenarios
for the asymptotic behavior of the Deffuant model on a connected graph as $t\rightarrow\infty$:

\begin{definition}
\begin{enumerate}[(i)]
\item {\itshape No consensus}\\
There will be finally blocked edges, i.e.\ edges $e=\langle u,v\rangle$ s.t.
$$|\eta_t(u)-\eta_t(v)|>\theta,$$
for all times $t$ large enough. Hence the vertices fall into different opinion groups.
\item {\itshape Weak consensus}\\
Every pair of neighbors $\{u,v\}$ will finally concur, i.e.
$$\lim_{t\to\infty}|\eta_t(u)-\eta_t(v)|=0.$$
\item {\itshape Strong consensus}\\
The value at every vertex converges, as $t\to\infty$, to a common limit $l$, where
$$l=\begin{cases}\text{the average of the initial opinion values},&\text{if }G\text{ is finite}\\
                 \E\eta_0,&\text{if }G\text{ is infinite.}\end{cases}$$

\end{enumerate}\vspace*{1em}
\noindent Let the scenario in which we have weak consensus, but at some vertices $v$
the value $\eta_t(v)$ is not converging be called {\itshape strictly weak
consensus}. Whether strictly weak consensus can actually occur (for some graphs and some initial
distributions) is an open problem.
\end{definition}

\noindent On finite graphs, strictly weak consensus is impossible as the opinion average is preserved
over time and in general the answer to the question whether we get consensus in the long run or not
clearly depends on the initial setting. With independent initial opinions distributed uniformly on $[0,1]$
even for values of $\theta$ close to but smaller than $1$ consensus might be prevented,
albeit with a small probability, e.g.\ when we get stuck right from the beginning with all the
opinions being close to either $0$ or $1$ leaving a gap larger than $\theta$ in between, preventing
any two individuals situated at different ends of the opinion range from compromising.
In the interdisciplinary area labelled ``sociophysics'' some work has been done in simulating
the long-term behavior of this model on various types of finite graphs, such as in
\cite{phys}.

On infinite regular lattices however, the picture is different and the minimal example almost
settled. For the graph on $\Z$ in which consecutive integers are joined by edges, Lanchier
\cite{Lanchier} showed for the standard case with i.i.d.\ {\upshape unif}$([0,1])$ distributed initial
values that regardless of $\mu$, which is just controlling the speed of convergence, the threshold
between no consensus and consensus $\theta_\text{c}$ is $\frac12$, which is the essence of
Theorem \ref{on Z}.

In this paper, we investigate what happens when this basic setting is generalized, in
two different directions. In Section 2 we stay on the one-dimensional lattice, i.e.\ the line
graph on $\Z$, but allow for more general initial distributions and are able to settle most
but not all cases of i.i.d.\ initial configurations (see Theorem \ref{gen}). We also generalize the
model slightly to allow for dependent initial opinions given by stationary ergodic sequences that satisfy
the so-called {\em finite energy condition}, known from percolation theory. (The generalization of the
Deffuant model to multivariate opinions can be found in the upcoming paper \cite{Deffuant2}.)

In Section 3, $\Z$ is replaced by the general regular lattice $\Z^d$. For $d\geq2$ most of the
techniques developed for the one-dimensional case $\Z$ break down, but we are at least able to show that
there won't be disagreement for a sufficiently large confidence bound, larger than $\tfrac34$ in the standard
i.i.d.\ uniform case (see Theorem \ref{Zd}). Furthermore, the arguments used transfer with only minor
changes to the more general case of an infinite, locally finite, transitive and amenable graph
(see Remark \ref{amenable}).

Finally, in the last section we consider the Deffuant model on the random subgraph of $\Z^d$ given by
supercritical i.i.d.\ bond percolation independent of the random variables driving the opinion dynamics,
i.e.\ the initial configuration and the Poisson processes. Besides an extension of the result we
derived for the full grid to this setting (Theorem \ref{perc1}), a lower bound for values of $\theta$ allowing
for strong consensus on the infinite component is established (Theorem \ref{perc2}).

We find it slightly surprising that we can prove this last result for supercritical percolation (with $p<1$)
but not for the full lattice. The more common situation for random processes living on supercritical percolation
clusters is that these are easier to handle on the full lattice.

\section{Generalized initial configurations on \texorpdfstring{$\Z$}{\bf Z}}\label{sec2}
\subsection{Independent and identically distributed initial opinion values}\label{1d}

\begin{theorem}[\bf Lanchier]\label{on Z}
    Consider the Deffuant model on the graph $(\Z,E)$, where $E=\{\langle v,v+1\rangle, v\in\Z\}$
    with i.i.d.\ {\upshape unif}$([0,1])$ initial configuration and fixed $\mu\in(0,\tfrac12]$.
    \begin{enumerate}[(i)]
    \item If $\theta>\tfrac12$, the model converges almost surely to strong consensus, i.e. with
    probability $1$ we have: $\lim_{t\to\infty}\eta_t(v)=\tfrac12$ for all $v\in\Z$.
    \item If $\theta<\tfrac12$ however, the integers a.s.\ split into (infinitely many) finite clusters
    of neighboring individuals asymptotically agreeing with one another, but no global consensus is approached.
    \end{enumerate}
\end{theorem}

\noindent
For the line graph, the critical value $\theta_\text{c}$ equals thus $\frac12$, but what happens at
criticality is still an open question. Lanchier's result was reproven by Häggström using somewhat more
basic techniques (see \cite{ShareDrink}, Thm.\ 6.5 and Thm.\ 5.2).

It turns out that the methods in \cite{ShareDrink} can be adapted to i.i.d.\ initial distributions beyond
the {\upshape unif}$([0,1])$ case. In the following theorem, we determine $\theta_\text{c}$ in all cases except
when the distribution's positive and negative parts both have infinite expectation (this case remains unsolved).
Upon completing this work, we learned that a similar extension was simultaneously and independently done by Shang
\cite{Shang}. Part (a) of our Theorem \ref{gen} conflicts with Thm.\ 1 in \cite{Shang}, the discrepancy
being due to Shang overlooking the crucial effect that gaps in the support of the distribution of $\eta_0$ have, if
they are large.

\begin{theorem}\label{gen}
Consider the Deffuant model on $\Z$ as described earlier with the only exception that the
initial opinions are not necessarily distributed uniformly on $[0,1]$ (but still i.i.d.).
\begin{enumerate}[(a)]
   \item Suppose the initial opinion of all the agents follows an arbitrary bounded distribution $\mathcal{L}(\eta_0)$
         with expected value $\E\eta_0$ and $[a,b]$ being the smallest closed interval containing its support.
         If $\E\eta_0$ does not lie in the support, there exists some maximal, open interval $I\subset[a,b]$ such that
         $\E\eta_0$ lies in $I$ and $\Prob(\eta_0\in I)=0$. In this case let $h$ denote the length of $I$, otherwise 
         set $h=0$.
         
         Then the critical value for $\theta$, where a phase transition from a.s.\ no consensus to a.s.\ strong
         consensus takes place, becomes $\theta_\text{\upshape c}=\max\{\E\eta_0-a,b-\E\eta_0,h\}$.
         The limit value in the supercritical regime is $\E\eta_0$.
   \item Suppose the initial opinions' distribution is unbounded but its expected value exists, either in the 
         strong sense, i.e.\ $\E\eta_0\in\R$, or the weak sense, i.e.\ $\E\eta_0\in\{-\infty,+\infty\}$.
         Then the Deffuant model with arbitrary fixed parameter $\theta\in(0,\infty)$ will a.s.\ behave
         subcritically, meaning that no consensus will be approached in the long run.
\end{enumerate}
\end{theorem}
  \noindent
   Before embarking on the proof of this generalized result, let us recall some key ingredients of the
   proof for the standard uniform case in \cite{ShareDrink}.
   The arguably most central among these is the idea of {\em flat points}. A vertex $v\in\Z$ is called $\epsilon$
   {\itshape -flat to the right} in the initial configuration $\{\eta_0(u)\}_{u\in\Z}$ if for all $n\geq0$:
   \begin{equation}\label{rflat}
     \frac{1}{n+1}\sum_{u=v}^{v+n}\eta_0(u)\in\left[\tfrac12-\epsilon,\tfrac12+\epsilon\right].
   \end{equation}
   It is called $\epsilon${\itshape-flat to the left} if the above condition is met with the sum running
   from $v-n$ to $v$ instead. Finally, $v$ is called {\itshape two-sidedly $\epsilon$-flat} if for all $m,n\geq0$
   \begin{equation}\label{tflat}
     \frac{1}{m+n+1}\sum_{u=v-m}^{v+n}\eta_0(u)\in\left[\tfrac12-\epsilon,\tfrac12+\epsilon\right].
   \end{equation}
   In order to grasp the crucial role of flat points another concept has to be mentioned, namely the representation
   of $\eta_t(v)$ as a weighted average of initial opinions (see La.\ 3.1 in \cite{ShareDrink}).
   This convex combination of initial opinions can be written in a neat form, using as a tool
   the non-random pairwise averaging procedure Häggström called {\em Sharing a drink} (SAD) in \cite{ShareDrink}.
   In the latter, one has an initial profile $\{\xi_0(v)\}_{v\in\Z}$, with $\xi_0(0)=1$ and $\xi_0(v)=0$ for all
   $v\neq0$, symbolizing a full glass of water at site $0$ and empty ones at all other sites. The averaging is now
   done as in (\ref{dynamics}) but without the threshold $\theta$ and the encounters are no longer random, but
   given by a sequence of edges. Elements of $[0,1]^\Z$ that can be obtained by a finite such sequence are called
   SAD-profiles. An appropriately tailored SAD-procedure will then mimick the dynamics of the corresponding 
   Deffuant model backwards in time in such a way that the state $\eta_t(0)$ in the Deffuant model at any given
   time $t>0$ can be written as a weighted average of states at time $0$ with weights given by an SAD-profile.
   In \cite{ShareDrink}, general properties of SAD-profiles and consequences for $\eta_t(0)$ are derived.
   For example, the opinion value at a vertex which is two-sidedly $\epsilon$-flat in the initial configuration can
   throughout time not move further away than $7\epsilon$ from its initial value (see La.\ 6.3 in \cite{ShareDrink}).
   
\vspace*{1em}
\begin{nproof}{of Theorem \ref{gen}}
\begin{enumerate}[(a)]
   \item The proof of this part will be subdivided into three steps marked by (i), (ii) and (iii).
     \begin{enumerate}
         \item[(i)]
         At first, let us suppose that the initial opinions are distributed on $[0,1]$ according to $\mathcal{L}(\eta_0)$
         having expected value $\E\eta_0=\tfrac12$ and mass around the expectation as well as
         at least one of the extremes, i.e.\ for all $\epsilon>0$ we have
         $$\Prob\left(\eta_0<\epsilon\ \text{ or }\eta_0>1-\epsilon\right)>0,\  
         \Prob\left(\tfrac12-\epsilon\leq\eta_0\leq\tfrac12+\epsilon\right)>0.$$
         Then we claim that the result of Theorem \ref{on Z} still holds true.
     \end{enumerate}  
         To prove this generalization of the standard uniform case is in fact to check that the crucial conditions in
         Häggström's \cite{ShareDrink} proof are met.
         First of all, the i.i.d.\ property guarantees that the distribution of the initial configuration
         is translation invariant, hence both the left- and right-shift of the system ($v\mapsto v-1\ \forall\,v\in\Z$
         and $v\mapsto v+1\ \forall\,v\in\Z$ respectively) are measure-preserving.

         The proof of La.\ 4.2 in \cite{ShareDrink} showing that $\Prob(v \text{ is }\epsilon\text{-flat to the right})>0$
         for every $\epsilon>0$ and $v\in\Z$ only uses the Strong Law of Large Numbers (SLLN), local modification
         (which employs that $\Prob\left(\tfrac12-\epsilon\leq\eta_0(v)\leq\tfrac12+\epsilon\right)>0$
         for all $\epsilon>0$, which we assumed) as well as $\E\eta_0=\tfrac12$.
      
         By symmetry the same is true for $\epsilon$-flatness to the left and the additional assumption that 
         $\Prob(\eta_0\notin[\epsilon,1-\epsilon])>0$ provides the missing ingredient to mimick Prop.\ 5.1 and 
         Thm.\ 5.2 in \cite{ShareDrink} verbatim: If $\theta<\tfrac12$, pick $\epsilon>0$ small enough such that 
         $\theta\leq\tfrac12-2\epsilon$.
         With positive probability any given site $v$ is prevented from ever compromising with its neighbors already
         by the initial configuration, namely if $v-1$ is $\epsilon$-flat to the left, $v+1$ $\epsilon$-flat to the
         right and $v$ itself an outlier in the sense that $\eta_0(v)\notin[\epsilon,1-\epsilon]$.
         This establishes the subcritical case {\itshape(i)} in Theorem \ref{on Z}.
      
         To show $\Prob(v \text{ is two-sidedly }\epsilon\text{-flat})>0$ for all $v\in\Z,\ \epsilon>0$ (in La.\ 4.3
         in \cite{ShareDrink})
         it is used once more that $\Prob\left(\tfrac12-\epsilon\leq\eta_0\leq\tfrac12+\epsilon\right)>0$. 
         Following the reasoning of Sect.\ 6 in \cite{ShareDrink} literally will settle the supercritical
         case. The only change that has to be made in order to adapt to the generalized setting is that the expected energy
         at time $t=0$, i.e.\ $\E(\eta_0(v)^{2})\in(0,1]$ in La.\ 6.2, is no longer $\tfrac13$ as for the uniform
         distribution. This minor change is not crucial however, since only the value's finiteness is used in the proof
         of Prop.\ 6.1.
     \begin{enumerate}  
         \item[(ii)] 
         Now suppose the initial distribution is as in (i), but fails to have mass around the expectation $\tfrac12$ and
         leaves a gap of width $h\in(0,1]$, i.e.\ there exists some maximal (open) interval $I\subset[0,1]$ of length $h$
         such that $\tfrac12$ lies in $I$ and $\Prob(\eta_0\in I)=0$. Then we claim that the critical value becomes
         $\theta_\text{\upshape c}=\max\{\tfrac12,h\}$.
     \end{enumerate}
      Changing the assumptions concerning the initial distribution of opinions as in (ii) will affect
      both the sub- and supercritical case as outlined in step (i). Clearly, the limiting behavior a.s.\ cannot
      be consensus for $\theta<h$ due to the fact that with probability $1$ we will have initial opinion values
      both below and above $\frac12$. Since an update, according to (\ref{dynamics}), can only take place between
      neighbors that are either both below or both above $\tfrac12$, sites with initial values above the gap $I$
      will throughout time stay above it and the same holds for initial values below the gap. In particular, edges
      that are blocked due to incident values lying on different sides of the gap $I$ in the beginning
      will stay blocked for ever, making consensus impossible. 
      
      For $\theta>h$, however, the behavior is pretty much as in the first case. Nevertheless, when it comes to
      show that there will be arbitrarily flat points with positive probability, one has to go about somewhat 
      differently due to the fact that for sufficiently small $\epsilon,\ 
      \Prob\left(\eta_0\in[\tfrac12-\epsilon,\tfrac12+\epsilon]\right)=0$, which implies that no site can be
      $\epsilon$-flat in the initial configuration by the very definition of flatness (taking $n=0$ in
      (\ref{rflat}) and $m=n=0$ in (\ref{tflat}) respectively).
      
      Let the gap interval be denoted by $I=(\alpha,\alpha+h)$ and fix $\delta>0$.
      Choose two rational numbers in 
      $[0,\tfrac12)\cap[\alpha-\delta,\alpha]$ and $(\tfrac12,1]\cap[\alpha+h,\alpha+h+\delta]$ respectively,
      say $p$ and $q$, and define $I_1:=[p,\alpha]$ and $I_2:=[\alpha+h,q]$.
      Since $I$ is maximal, one can choose these rationals in such a way that
      $$\Prob(\eta_0\in I_1)>0 \text{ as well as } \Prob(\eta_0\in I_2)>0.$$
      \vspace{-0.8cm}
    \begin{figure}[H]
     \hspace*{6.6cm}
     \unitlength=0.80mm
     \begin{picture}(70.00,20.00)
          \put(-2.00,0.00){\vector(1,0){68.00}}
          \put(0.00,-1.50){\line(0,1){3.00}}
          \put(30.00,-1.50){\line(0,1){3.00}}
          \put(60.00,-1.50){\line(0,1){3.00}}
          \put(-0.80,-6.00){$\scriptstyle 0$}
          \put(28.70,-7.00){$\scriptscriptstyle\tfrac12$}
          \put(59.20,-6.00){$\scriptstyle 1$}
          \put(42.00,-1.50){\line(0,1){3.00}}
          \put(10.00,-1.50){\line(0,1){3.00}}
          \put(30.00,9.00){\vector(1,0){11.50}}
          \put(30.00,9.00){\vector(-1,0){19.50}}
          {\color[rgb]{.3,.3,1} 
          \put(7.00,4.00){$I_1$}
          \put(42.30,4.00){$I_2$}}
          \put(8.70,-6.00){$\scriptstyle \alpha$}
          \put(38.10,-6.00){$\scriptstyle \alpha+h$}
          {\color[rgb]{.6,.6,1} 
          \linethickness{1.5pt}
          \put(7.50,0.00){\line(1,0){2.50}}
          \put(42.00,0.00){\line(1,0){4.00}}}
          \put(25.00,11.00){$I$}
          {\color[rgb]{0,0,1} 
          \linethickness{1pt}
          \put(7.50,-1.50){\line(0,1){3.00}}
          \put(46.00,-1.50){\line(0,1){3.00}}
          \put(5.50,-4.00){$\scriptstyle p$}
          \put(46.00,-4.00){$\scriptstyle q$}}
     \end{picture}
   \end{figure}\vspace{-2.22cm}
      
      \par
      \begingroup
      \rightskip17.7em
      Clearly, there exist natural numbers $m,n$ s.t.\ $\tfrac{m}{m+n}\,p+\tfrac{n}{m+n}\,q=\frac12$.
      As numbers from $I_1$ and $I_2$ differ not more than $\delta$ from $p$ and $q$ respectively, the average
      of $m$ numbers from $I_1$ and $n$ numbers from $I_2$ surely lies within $[\tfrac12-\delta,\tfrac12+\delta]$.
      \par\endgroup
      Thus, we get that for any fixed $k\in\N=\{1,2,\dots\}$:
      \begin{equation}\label{innerpart}
      \Prob\left(\frac{1}{k(m+n)}\sum_{v=0}^{k(m+n)-1}\eta_0(v)\in\left[\tfrac12-\delta,\tfrac12+\delta\right]\right)>0.
      \end{equation}
      
      Now let us consider some fixed time point $t>0$ and the corresponding configuration $\{\eta_t(v)\}_{v\in\Z}$.
      There is a.s.\ an infinite increasing sequence of not necessarily consecutive edges
      $(\langle v_k,v_k+1\rangle)_{k\in\N}$ to the right of site $0$, on which no Poisson event has occurred up to
      time $t$.
      
      Clearly, their positions are random,
      so let $l_k:=v_{k+1}-v_k,\text{ for } k\in\N,$ denote the random lengths of the intervals in between and
      $l_0:=v_1-v_0+1$ the one of the interval including $0$, where $\langle v_0-1,v_0\rangle$ is the first 
      edge to the left of the origin without Poisson event. Since the involved Poisson processes are independent,
      it is easy to verify that the $l_k,\ k\in\N_0=\{0,1,2,\dots\}$, are i.i.d., having a geometric distribution
      on $\N$ with parameter $\text{e}^{-t}$.
      
      For $\delta>0$, let $A_\delta$ be the event that $l_0$ is finite and only finitely many of the events
      $\{l_k\geq k\delta\},\ k\in\N,$ occur. Then their independence and the Borel-Cantelli-Lemma tell us 
      that $A_\delta$ has probability $1$. On $A_\delta$ however the following holds a.s.\ true:\vspace{-0.5cm}
      
      \begin{align*}
       \limsup_{v\to\infty}\frac{1}{v+1}\sum_{u=0}^{v}\eta_t(u)&=\limsup_{v\to\infty}\frac{1}{v+1}\sum_{u=v_0}^{v}\eta_t(u)\\
       &\leq\limsup_{v\to\infty}\frac{1}{v+1}\sum_{u=v_0}^{v}\eta_0(u)+\delta\\
       &=\lim_{v\to\infty}\frac{1}{v+1}\sum_{u=0}^{v}\eta_0(u)+\delta=\frac12+\delta.
      \end{align*}
 
      The inequality follows from the fact that the Deffuant model is mass-preserving in the sense that
      $\eta_t(u)+\eta_t(v)=\eta_{t-}(u)+\eta_{t-}(v)$ in (\ref{dynamics}), hence for all $k\in\N$:
      $\sum_{u=v_0}^{v_k}\eta_0(u)=\sum_{u=v_0}^{v_k}\eta_t(u)$. For the average at time $t$ running from 
      $v_0$ to some $v\in\{v_k+1,\dots, v_{k+1}\}$ to differ by more than $\delta$ from the one at time 0,
      the interval has to be of length more than $k\delta$, since $v_k\geq k$ and $\eta_t(u)\in[0,1]$ for all $t,u$.
      This, however, will happen only finitely many times.
      Since $\delta$ was arbitrary and mimicking the same argument for the limes inferior, we have established
      that
      \begin{equation}\label{outerpart}
       \lim_{v\to\infty}\frac{1}{v+1}\sum_{u=0}^{v}\eta_t(u)=\frac12 \text{ almost surely.}
      \end{equation}
      Now fix $\epsilon>0$ such that $h+\tfrac{\epsilon}{3}<\theta$, choose $\delta=\tfrac{\epsilon}{6}$ in 
      (\ref{innerpart}) as well as the rationals $p,q$ and integers $m,n$ accordingly. Due to (\ref{outerpart})
      there exists some integer number $k$ s.t.\ the event
      $$A:=\left\{\frac{1}{v+1}\sum_{u=0}^{v}\eta_t(u)\in\left[\tfrac12-\tfrac{\epsilon}{3},
      \tfrac12+\tfrac{\epsilon}{3}\right]\text{ for all }v\geq N\right\}$$
      has probability greater than $1-\text{e}^{-2t}$, where $N:=k(m+n)-1$. 
      Let $B$ in turn be the event that there was no Poisson event on $\langle -1,0\rangle$ and 
      $\langle N,N+1\rangle$ up to time $t$, hence $\Prob(B)=\text{e}^{-2t}$. Finally, let $C$ be the event
      that the initial values $\eta_0(0),\dots,\eta_0(N)$ were all in $[p,q],\ km$ of them below $\tfrac12$,
      $kn$ above $\tfrac12$, and the Poisson firings on the edges $\langle 0,1\rangle,\dots,\langle N-1,N\rangle$
      up to time $t$ are sufficiently numerous such that, given $B$,
      $\eta_t(u)\in[\tfrac12-\tfrac{\epsilon}{3},\tfrac12+\tfrac{\epsilon}{3}]$
      for all $u\in\{0,\dots,N\}$. Note that $q-p\leq h+2\delta<\theta$, hence every such Poisson event will lead
      to an update, and that the independence of the initial configuration and the Poisson processes together
      with the considerations leading to (\ref{innerpart}) imply that $C$ has positive probability. Furthermore,
      $C$ is independent of $B$ and $A\cap B$ cannot have probability $0$, since
      $$\Prob(A\cap B)=\Prob(A)+\Prob(B)-\Prob(A\cup B)>(1-\text{e}^{-2t})+\text{e}^{-2t}-\Prob(A\cup B)\geq0.$$
      This gives that the conditional probabilities $\Prob(A|B)$ and $\Prob(C|B)$ are both strictly greater than $0$.
      
      Given $B$, we can apply the coupling trick, commonly known as
      {\itshape local modification}, precisely as in the proof of La.\ 4.2 in \cite{ShareDrink} to find
      that $\Prob(A\cap B\cap C)>0$. A one-line calculation shows that $A\cap B\cap C$ implies the $\epsilon$-flatness
      to the right of site $0$ in the configuration at time $t$.
      
      Since the distribution of $\{\eta_t(u)\}_{u\in\Z}$ is still translation and left-right reflection
      invariant, every site $v\in\Z$ is $\epsilon$-flat to the right (or left) at time $t$ with positive
      probability on the one hand, and on the other this allows us to follow the argument in (i) settling the
      subcritical case and forcing $\theta_\text{c}\geq\max\{\tfrac12,h\}$.
      
      A short moment's thought verifies that $\epsilon$-flatness to the right of site $v$ and $\epsilon$-flatness
      to the left of site $v-1$ simultaneously imply two-sided $\epsilon$-flatness of both, $v$ and $v-1$. Let
      $A^r_v, B^r_v,C^r_v$ be the sets appearing above, corresponding to site $v$ and ``right'', and
      $A^l_{v-1}, B^l_{v-1},C^l_{v-1}$ the ones corresponding to $v-1$ and ``left''. The involved independences lead to
      \begin{align*}
       \Prob(A^r_v\cap B^r_v\cap C^r_v\cap A^l_{v-1}\cap B^l_{v-1}\cap C^l_{v-1})\\
       &\hspace{-5.5cm}=\Prob(A^r_v\cap C^r_v\cap A^l_{v-1}\cap C^l_{v-1}|B^r_v\cap B^l_{v-1})\cdot
         \Prob(B^r_v\cap B^l_{v-1})\\
       &\hspace{-5.5cm}=\Prob(A^r_v\cap C^r_v|B^r_v\cap B^l_{v-1})\cdot
         \Prob(A^l_{v-1}\cap C^l_{v-1}|B^r_v\cap B^l_{v-1})\cdot\Prob(B^r_v\cap B^l_{v-1})\\
       &\hspace{-5.5cm}=\Prob(A^r_v\cap C^r_v|B^r_v)\cdot\Prob(A^l_{v-1}\cap C^l_{v-1}|B^l_{v-1})\cdot
        \Prob(B^r_v\cap B^l_{v-1})>0,
      \end{align*} 
      since $\Prob(B^r_v\cap B^l_{v-1})=\text{e}^{-3t}>0$. Hence two-sided $\epsilon$-flatness at time $t$
      has positive probability as well. Following the argument corresponding to the supercritical case in (i),
      using the preserved translation invariance of the distribution of $\{\eta_t(u)\}_{u\in\Z}$ once more,
      we find that there will be consensus in the long run, if only $\theta>\max\{\tfrac12,h\}$. Putting both
      arguments together, this proves the claim $\theta_\text{c}=\max\{\tfrac12,h\}$.
      
     \begin{enumerate}
         \item[(iii)]
         Finally, suppose $[a,b]$ is the smallest closed interval containing the support of the initial opinions'
         distribution and the latter features a gap of width $h\in[0,b-a]$ around the expected value $\E\eta_0\in[a,b]$.
         Then we claim that the critical value becomes $\theta_\text{\upshape c}=\max\{\E\eta_0-a,b-\E\eta_0,h\}$
         and the limit in the case of strong consensus is $\E\eta_0$.
     \end{enumerate} 
      Clearly, the dynamics of the Deffuant model are not effected by translations ($x\mapsto x+c$ for
      some constant $c\in\R$) of the initial distribution. A scaling ($x\mapsto\tfrac{x}{c},\ c\in\R_{>0}$) has the
      only effect that the value for the parameter $\theta$ has to be rescaled too, in order to get identical
      dynamics.
      
      Let $c:=\max\{\E\eta_0-a,b-\E\eta_0\}$ and consider the linear transformation
      $$x\mapsto\tfrac{x-\E\eta_0}{2\,c}+\tfrac12.$$
      The transformed initial distribution satisfies the assumptions in step (ii) and leaves a gap of width
      $\tfrac{h}{2\,c}$ around the mean $\tfrac12$. Therefore, the considerations in (ii) allow us to conclude 
      $$\theta_\text{c}=2\,c\cdot\max\{\tfrac12,\tfrac{h}{2\,c}\}=\max\{c,h\}=\max\{\E\eta_0-a,b-\E\eta_0,h\}.$$
         
      Note that the limit of an individual opinion in the supercritical case is the retransformed equivalent
      of $\frac12$, i.e.\ $2\,c\cdot\big(\tfrac12+(\tfrac{\E\eta_0}{2\,c}-\tfrac12)\big)=\E\eta_0$. 
     
 \item To prove the statement on unbounded initial distributions we have to treat two cases, namely the one
       where $\E|\eta_0|<\infty$ and the other where exactly one of both $\E\eta_0^+,\E\eta_0^-$ is infinite.
     \begin{enumerate}
      \item[(i)] 
      In case of an unbounded initial distribution with existing first moment and expectation $\E\eta_0<\infty$,
      the SLLN reads (for arbitrarily chosen $v\in\Z$):
      \begin{equation*}
       \Prob\left(\lim_{n\to\infty}\frac{1}{n+1}\sum_{u=v}^{v+n}\eta_0(u)=\E\eta_0\right)=1.
      \end{equation*}
      Consequently, there exists some number $r>0$ s.t.
      \begin{equation*}
       \Prob\left(\frac{1}{n+1}\sum_{u=v}^{v+n}\eta_0(u)\in[\E\eta_0-r,\E\eta_0+r]\text{ for all }n\in\N_0\right)>0.
      \end{equation*}
      Slightly abusing the definition (the expectation $\tfrac12$ in (\ref{rflat}) would have to be replaced
      by $\E\eta_0$), one could say that with positive probability site $v$ is $r$-flat to the right.
      
      Let the confidence bound $\theta$ take on some value in $(0,\infty)$. Strictly along the lines of
      Prop.\ 5.1 in \cite{ShareDrink}, it follows that if $v-1$ and $v+1$ are $r$-flat to the left and
      right respectively and simultaneously $\eta_0(v)\notin[\E\eta_0-r-\theta,\E\eta_0+r+\theta]$ -- an event with
      positive probability -- the values at $v-1$ and $v+1$ will throughout all of time stay within the interval
      $[\E\eta_0-r, \E\eta_0+r]$ leaving the edges $\langle v-1,v\rangle$ and $\langle v,v+1\rangle$ blocked. Since this
      happens at every site $v$ with positive probability, ergodic theory tells us that it will almost surely
      occur at infinitely many sites.
     \end{enumerate}
     \begin{enumerate}
      \item[(ii)]
      Now suppose that the expectation of $\eta_0$ exists only in the weak sense, i.e.\ 
      $\E\eta_0\in\{-\infty,+\infty\}$. Once more, symmetry allows us to focus on the case
      $\E\eta_0^+=\infty,\ \E\eta_0^-<\infty$. In this case the SLLN reads
      \begin{equation}\label{SLLN}
       \Prob\left(\lim_{n\to\infty}\frac{1}{n}\sum_{u=v+1}^{v+n}\eta_0(u)=\infty\right)=1.
      \end{equation}
      We can assume $\Prob(\eta_0<0)>0$, otherwise a translation (irrelevant for the dynamics) as in the last
      step of (a) will reduce the problem to this setting. Some one-sided version of the idea of proof using flatness
      can then be employed.
      
      Let the confidence bound $\theta\in(0,\infty)$ be arbitrary but fixed. By (\ref{SLLN}), for sufficiently 
      large $N\in\N$ the following event has non-zero probability:
      \begin{equation*}
       A_N:=\left\{\frac{1}{n}\sum_{u=v+1}^{v+n}\eta_0(u)>\theta\text{ for all }n\geq N\right\}.
      \end{equation*}
      Local modification is again the key step to advance. Let $\xi:=\mathcal{L}(\eta_0)$ denote the 
      distribution of $\eta_0$ and $\xi|_{(\theta, \infty)}$ its distribution conditioned on the
      event $\{\eta_0>\theta\}$. Clearly, $\xi$ is stochastically dominated by $\xi|_{(\theta, \infty)}$,
      i.e.\ $\xi\preceq\xi|_{(\theta, \infty)}$, implying 
      $$\mathcal{L}\big((\eta_0(u))_{u\geq v+1}\big)=\bigotimes_{u\geq v+1}\xi\preceq 
      \left(\bigotimes_{u=v+1}^{v+N}\xi|_{(\theta, \infty)}\right)\otimes\left(\bigotimes_{u>v+N}\xi\right).$$
      
      Let $B$ be the event $\{\eta_0(v+1)>\theta,\dots,\eta_0(v+N)>\theta\}$, which has non-zero probability, and
      $$A_1:=\left\{\frac{1}{n}\sum_{u=v+1}^{v+n}\eta_0(u)>\theta\text{ for all }n\in\N\right\}.$$
      The stochastic domination from above yields:
      \begin{align*}
       \Prob(A_1)&\geq\Prob(A_1\cap B)=\Prob(A_N\cap B)=\Prob(A_N|B)\cdot\Prob(B)\\
                 &\geq\Prob(A_N)\cdot\Prob(B)>0.
      \end{align*}
      The very same ideas as in the proof of Prop.\ 5.1 in \cite{ShareDrink} show that if $A_1$ occurs and the edge
      $\langle v,v+1\rangle$ doesn't allow for an update, irrespectively of the dynamics on $\{u\in\Z, u\geq v+1\}$,
      we have that $\eta_t(v+1)>\theta$ is preserved for all times $t>0$. By symmetry the same holds for site $v-1$
      and the half-line to the left, i.e.\ $\{u\in\Z, u\leq v-1\}$. Independence of the initial opinions
      therefore guarantees that with positive probability, the initial configuration can be such that $\eta_0(v)<0$ 
      and the values at sites $v-1$ and $v+1$ are doomed to stay above $\theta$, blocking the edges adjacent
      to $v$ once and for all.
      Ergodicity makes sure that with probability $1$ infinitely many sites will get stuck this way.
      \end{enumerate}
\end{enumerate}\vspace*{-0.17cm}
\end{nproof}\vspace*{-1.5em}

\begin{examples}
\begin{enumerate}[(a)]
 \item As a first toy application of the above result, let us consider the Deffuant model on $\Z$ in which the 
       initial values are independently distributed according to a beta distribution Beta$(\alpha, \beta)$, where
       the two real numbers $\alpha,\beta> 0$ represent the parameters of this family of distributions.
       That means $\eta_0$ has support $[0,1]$ and its distribution the density function
       $$f_{\alpha,\beta}(x)=\frac{1}{\text{B}(\alpha, \beta)}\;x^{\alpha-1}\,(1-x)^{\beta-1},\quad\text{for }x\in[0,1],$$
       where the normalizing factor is given by the beta function
       $$\text{B}(\alpha, \beta)=\int_0^1 t^{\alpha-1}\,(1-t)^{\beta-1}\,dt.$$
       Since $f_{\alpha,\beta}>0$ on the open interval $(0,1)$, there are no gaps in the support and a simple calculation
       shows $\E\eta_0=\tfrac{\alpha}{\alpha+\beta}$. Consequently, part (a) of Theorem \ref{gen} shows that the critical
       value for the confidence bound separating the regimes of consensus and fragmentation is
       $$\theta_\text{c}=\begin{cases}\tfrac{\alpha}{\alpha+\beta},&\text{if }\alpha\geq\beta\\
                                      \tfrac{\beta}{\alpha+\beta},&\text{otherwise} \end{cases}
                                      =\frac{\max\{\alpha,\beta\}}{\alpha+\beta}.$$
       This example appears in \cite{Shang} as well.
 \item Letting the initial values be independently drawn from a uniform distribution on the discrete set
       $\{-0.8,-0.3,0.7,0.8\}$, $[-0.8,0.8]$ is the minimal closed interval containing the support of
       $\mathcal{L}(\eta_0)$. 
       Obviously, there is a gap of width $h=1$ around the mean $\E\eta_0=0.1$. Part (a) of Theorem \ref{gen}
       tells us that $\theta_\text{c}=\max\{\E\eta_0-(-0.8),0.8-\E\eta_0,h\}=\max\{0.9,0.7,1\}=1$.
 \item If we take the initial opinions to be i.i.d.\ and uniform on the set $[0,\tfrac18]\cup[\tfrac78,1]$ instead, its
       expectation is $\E\eta_0=\tfrac12$. But even though $\Prob(|\eta_0-\E\eta_0|>\tfrac12)=0$, a choice of 
       $\theta\in(\tfrac12,\tfrac34)$ will a.s.\ lead to no consensus, as $\theta_\text{c}=\tfrac34$, again by part
       (a) of the above theorem. The next proposition actually shows that even for $\theta=\theta_\text{c}$ the
       limiting scenario will a.s.\ be no consensus.
\end{enumerate}
\end{examples}

\noindent
For a bounded initial distribution whose support has a large gap around its mean, we can
deal with the behavior at criticality:

\begin{proposition}\label{crit}
 Let the initial opinions be again i.i.d.\ with $[a,b]$ being the smallest closed interval containing
 the support of the marginal distribution,
 and the latter feature a gap $(\alpha,\beta)$ of width $\beta-\alpha>\max\{\E\eta_0-a,b-\E\eta_0\}$ around its
 expected value $\E\eta_0\in[a,b]$.\vspace*{0.5em}

\noindent At criticality, that is for $\theta=\theta_\text{\upshape c}
 =\max\{\E\eta_0-a,b-\E\eta_0,\beta-\alpha\}=\beta-\alpha$, we get
 the following: If both $\alpha$ and $\beta$ are atoms of the distribution $\mathcal{L}(\eta_0)$, i.e.\ 
 $\Prob(\eta_0=\alpha)>0$ and $\Prob(\eta_0=\beta)>0$, the system approaches a.s.\ strong consensus. However, it
 will a.s.\ lead to no consensus if either $\Prob(\eta_0=\alpha)=0$ or $\Prob(\eta_0=\beta)=0$.
\end{proposition}

\begin{proof}
In order to prove this statement, we can follow the arguments in the proof of part (a) of Theorem \ref{gen}.
By the translation and scaling invariance of the dynamics as described in step (iii) of the cited proof, we can
restrict ourselves to the case in step (ii) and assume that the support of $\mathcal{L}(\eta_0)$ is a subset of
$[0,1]$, $\E\eta_0=\tfrac12$ and $\Prob\left(\eta_0<\epsilon\ \text{ or }\eta_0>1-\epsilon\right)>0$ for all
$\epsilon>0$. Note that under these further assumptions, we have $\theta=\theta_\text{\upshape c}=\beta-\alpha>\tfrac12$.

If both ends of the gap are atoms, we can follow the reasoning of the supercritical case in (ii)
and for every $\delta>0$ choose natural numbers $m,n$ such that
$\tfrac{m}{m+n}\,\alpha+\tfrac{n}{m+n}\,\beta\in[\tfrac12-\delta,\tfrac12+\delta]$, to get (\ref{innerpart}).
Using such a collection of initial opinions, i.e.\ $m$ times the value $\alpha$ and $n$ times $\beta$, all of
them will be precisely within the confidence bound, hence allow for the manipulation described above as local
modification. Having arbitrarily flat points with positive probability at time $t>0$, $\theta>\tfrac12$
guarantees a.s.\ strong consensus.

The negative statement is easy to handle. If, without loss of generality, $\Prob(\eta_0=\alpha)=0$, with
probability 1 there will
be no initial value lying in the interval $[\alpha,\beta)$. Since $\theta=\beta-\alpha$, this gap cannot be
bridged. We refer once more to step (ii) in the proof of part (a) of Theorem \ref{gen} for a more detailed reasoning.
\end{proof}\vspace*{1em}

\noindent Does Proposition \ref{crit} constitute progress in the attempt to solve the critical case in the
setting of uniformly distributed initial opinions (the open problem mentioned right after Theorem \ref{on Z})?
Probably not, since here, due to the large width of the gap $\beta-\alpha>\max\{\E\eta_0-a,b-\E\eta_0\}$, the
criticality comes only from the gap in the distribution, not the distance between the mean and the extreme
ends of the initial distribution.

As already mentioned in the introductory section, a next step of generalization in terms of the initial
opinions would be vector-valued distributions. Despite the fact that this seems to be a minor modification
it invokes major changes and would thus excessively expand this section, which is why it is omitted here and
treated as a separate topic in \cite{Deffuant2}.

\subsection{Dependent initial opinion values}\label{dep}

The definition of the Deffuant model generalizes straightforwardly to dependent initial configurations.
Considering that -- in our treatment of the model on $\Z$ in the foregoing subsection -- the independence of initial
opinions was merely used to deduce translation invariance and ergodicity with respect to shifts as well as for
the local modification, it is a valid question in how far the results of Theorem \ref{gen} can be generalized
to initial configurations $\{\eta_0(v)\}_{v\in\Z}$ that do not form an i.i.d.\ sequence.
The example below shows that stationarity and ergodicity of the
sequence of initial opinions is not enough to retain the results from Subsection \ref{1d}. In order to be
able to locally modify the configuration as done in the proof of Theorem \ref{gen}, we have to add an extra
condition, which is a natural extension to continuous state spaces of the well-known finite energy condition
of percolation theory (see for instance Def.\ 2 in \cite{BK}).

\begin{definition}\label{finen}
Let $\{\xi_v\}_{v\in\Z}$ be a stationary sequence of random variables. It is said to satisfy the
{\em finite energy condition} if it allows conditional probabilities such that the conditional distribution
of $\xi_0$ given $\{\xi_v\}_{v\in\Z\setminus\{0\}}$ almost surely has the same support as the marginal distribution
$\mathcal{L}(\xi_0)$.
\end{definition}

\noindent
Carefully checking its proof with this extra condition in hand, we can get the following generalization of
Theorem \ref{gen}:

\begin{theorem}
Consider the Deffuant model on $\Z$ with initial opinions values $\{\eta_0(v)\}_{v\in\Z}$.
If $\{\eta_0(v)\}_{v\in\Z}$ is a stationary sequence of random variables, ergodic with respect to
shifts and satisfying the finite energy condition, the results of Theorem \ref{gen} still hold true.
\end{theorem}

\noindent
To see that the added assumption that conditioning on the configuration apart from a given
site $v$ will not change the support of the distribution at site $v$ is essential and can not be dropped,
see the following example.

\begin{example}
Let $U$ be a random variable, uniformly distributed on $\{-4,-3,\dots,4\}$. The initial configuration
will now be made up of blocks of length $9$ centered in the sites $\{c_k\}_{k\in\Z}:=\{U+9\,k\}_{k\in\Z}$.
Each block will independently be either of the form $\eta_0(c_k)=\tfrac12$ and $\eta_0(v)=0$ for 
$v\in\{c_k-4,\dots,c_k-1,c_k+1,\dots,c_k+4\}$ or $\eta_0(c_k)=\tfrac12$ and $\eta_0(v)=1$ for 
$v\in\{c_k-4,\dots,c_k-1,c_k+1,\dots,c_k+4\}$, both with probability $\tfrac12$.

The initial configuration $\{\eta_0(v)\}_{v\in\Z}$ defined in that way is translation invariant and
ergodic with respect to shifts, having the marginal distribution $\mathcal{L}(\eta_0)$, where
$\Prob(\eta_0=0)=\Prob(\eta_0=1)=\tfrac{4}{9}$ and $\Prob(\eta_0=\tfrac12)=\tfrac{1}{9}$.

If Theorem \ref{1d} applied, the critical value should be $\theta_\text{\upshape c}=\tfrac12$ but it is
not hard to see that for $\theta<\tfrac45$ compromises are at first confined to happen within intervals consisting
of blocks of the same kind and can thus only lead to values in $[0,\tfrac{1}{10}]\cup[\tfrac{9}{10},1]$ at sites next
to a neighboring block of the other kind, see also Thm.\ 2.3 in \cite{ShareDrink}. This means that the edges connecting
two blocks of different kind will be blocked throughout time forcing a.s.\ no consensus.
\par
\vspace{0.5em}
\noindent
Due to the fixed block size, the sequence $\{\eta_0(v)\}_{v\in\Z}$ as defined above is obviously not mixing.
An easy modification, for instance allowing random block lengths taking values 9 and 11, shows that even an initial
configuration which is given by a stationary mixing sequence of random variables does not, in general, allow for
the results of the i.i.d.\ case to be transferred.
%
\end{example}

\section{Upper bound for the critical range of \texorpdfstring{$\theta$}{theta} on 
\texorpdfstring{$\Z^d$}{{\bf Z\textasciicircum d}}}\label{UpperBound}
\subsection{Application of energy arguments}

Moving on to higher dimensions as far as the underlying lattice is concerned provides the opportunity to go around
blocked edges and there is no handy generalization of the notion of flatness. Among other things, these changes render
most of the arguments used in the $\Z$ case void. Enough can be resurrected, however, to establish a lower bound for
$\theta$ above which consensus is achieved. Throughout Sections 3 and 4 (Theorem \ref{perc2} being an exception) we will
only assume that the configuration of initial opinion values $\{\eta_0(v)\}_{v\in\Z^d}$ is stationary and ergodic with
respect to shifts of the kind $T_i: v\mapsto v+e_i$, where $e_i$ is the $i$th standard basis vector of $\R^d$ for
$i\in\{1,\dots,d\}$.

\begin{theorem}\label{Zd}
\begin{enumerate}[(a)]
\item  If the initial values are distributed uniformly on $[0,1]$ and $\theta>\tfrac34$,
       the configuration will a.s.\ approach weak consensus, i.e.
       $$\Prob\big(\lim_{t\to\infty}|\eta_t(u)-\eta_t(v)|=0\big)=1$$
       for all $u,v\in\Z^d$ s.t.\ $\langle u,v\rangle$ forms an edge.
\item  For general initial distributions on $[0,1]$ the range of $\theta$, where final
       consensus is guaranteed, is non-trivial, i.e.\ including values smaller than $1$,
       unless the initial values are concentrated on $0$ and $1$, taking on both values
       with positive probability.
\end{enumerate}
\end{theorem}

\noindent
     To prove this, we need first to establish some lemmas, the first one involving the idea of {\em energy},
     introduced in Sect.\ 6 of \cite{ShareDrink} (not to be confused with the completely unrelated concept
     of finite energy from Subsection \ref{dep}). 
     
     Assume the initial values $\{\eta_0(v)\}_{v\in\Z^d}$ have a stationary distribution, ergodic with respect to shifts
     and the marginal distribution has bounded support, without loss of generality we can take $[0,b]$
     to be the smallest closed interval containing it. Denote by
     $W_t(v)=\En(\eta_t(v))$ the energy at vertex $v$ at time $t$, where $\En:[0,b]\to\R_{\geq0}$
     is some fixed convex function.
     If a Poisson event occurs at the edge $e=\langle u,v\rangle$ at time $t$, and the values at $u$ and $v$,
     $\eta_{t-}(u)$ and $\eta_{t-}(v)$ respectively, are within $\theta$, energy is transferred and (possibly) lost
     along the edge. The latter to the amount
     \begin{equation}\label{loss}
      w_t(e):=(W_{t-}(u)+W_{t-}(v))-(W_t(u)+W_t(v)).
     \end{equation}
     Since $\eta_t(u)=(1-\mu)\,\eta_{t-}(u)+\mu\,\eta_{t-}(v)$ and $\eta_t(v)=(1-\mu)\,\eta_{t-}(v)+\mu\,\eta_{t-}(u)$,
     the convexity of $\En$ gives:
     \begin{align*}
     W_t(u)+W_t(v) & \leq (1-\mu)\,W_{t-}(u)+\mu\,W_{t-}(v)+(1-\mu)\,W_{t-}(v)+\mu\,W_{t-}(u)\\
                   & = W_{t-}(v)+W_{t-}(u),
     \end{align*}             
     i.e.\ the non-negativity of $w_t(e)$. Let $T$ denote the sequence of arrival times of the Poisson events at $e$
     and define the accumulated energy loss along $e$ as $$W_t^\text{loss}(e):=\sum_{s\in T\cap[0,t]}w_s(e).$$
     Finally, let $E(v)$ denote the set of edges incident to $v$ and define the total energy attributed to vertex $v$ as
     \begin{equation}\label{tot}
     W^\text{tot}_t(v):=W_t(v)+\frac12\,\sum_{e\in E(v)}W_t^\text{loss}(e).
     \end{equation}
     Note that by (\ref{loss}) the sum $W^\text{tot}_t(v)+W^\text{tot}_t(u)$
     is preserved when an update along the edge $\langle u,v\rangle$ takes place.
     Along the lines of La.\ 6.2 in \cite{ShareDrink} we can show the following analog:

\begin{lemma}\label{avpre}
For every $v\in\Z^d$ and $t\geq0$ we have
  \begin{equation}\label{ergodic}
     \E[W^\text{\upshape tot}_t(v)]=\E[W_0(\mathbf{0})].
  \end{equation}
\end{lemma}

\begin{proof}
     Note first that for fixed time $t$ the process $\{W^\text{tot}_t(v)\}_{v\in\Z^d}$ only depends on the initial
     configuration and the independent Poisson processes attributed to the edges. Its distribution is therefore
     translation invariant and the process ergodic with respect to shifts.
     
     Let $\Lambda_n=[-n,n]^d$ denote the box of sidelength $2n$ centered at the origin $\mathbf{0}$. 
     It contains $|\Lambda_n|=(2n+1)^d$ vertices of the grid $\Z^d$ and there are $2d\,(2n+1)^{d-1}$ edges linking
     vertices inside $\Lambda_n$ to vertices outside of the box. The set of such edges is called {\em edge boundary}
     of $\Lambda_n$ and denoted by $\partial_E\Lambda_n$.
     
     The multivariate version of Birkhoff's Theorem, attributed to Zygmund (see e.g.\ Thm.\ 10.12 in \cite{Foundations}),
     tells us that
     \begin{equation}\label{Birkhoff}
      \lim_{n\to\infty}\frac{1}{|\Lambda_n|}\sum_{v\in\Lambda_n}W^\text{tot}_t(v)=\E[W^\text{tot}_t(\mathbf{0})] 
      \text{  almost surely.}
     \end{equation}
     Note that the statement of (\ref{Birkhoff}) is still true if we pass from the original sequence of sets
     $(\Lambda_n)_{n\in\N}$ to any subsequence.
     
     Translation invariance of the configuration implies $\E[W^\text{tot}_t(v)]=\E[W^\text{tot}_t(\mathbf{0})]$
     for all sites $v$ and by definition $W_0^\text{loss}(e)=0$ for all edges $e$ since at time 0 no Poisson
     event has occurred yet, hence $W^\text{tot}_0(\mathbf{0})=W_0(\mathbf{0})$.
     \par\vspace{0.5em}
     \noindent
     Let us now choose a subsequence $(\Lambda_{n_k})_{k\in\N}$ such that
     \begin{equation}\label{summable}
     \sum_{k=1}^\infty\frac{|\partial_E\Lambda_{n_k}|}{|\Lambda_{n_k}|}<\infty.
     \end{equation}
     As mentioned, (\ref{Birkhoff}) clearly implies
     \begin{equation}\label{subseq}
      \lim_{k\to\infty}\frac{1}{|\Lambda_{n_k}|}\sum_{v\in\Lambda_{n_k}}W^\text{tot}_t(v)=\E[W^\text{tot}_t(\mathbf{0})] 
      \text{  almost surely.}
     \end{equation}
     In order to establish the claim it is therefore left to show that the limit in (\ref{subseq}) is constant over time.
      
     Take $\epsilon>0$ small and fix a time interval $[t,t+\epsilon]$. Note that the energy function $\En$ is
     bounded on $[0,b]$ by $M:=\max\{\En(0),\En(b)\}$, due to its convexity.
     Let $N_{n,\epsilon}$ be the number of Poisson events on edges in $\partial_E\Lambda_{n}$ within the time interval
     $(t,t+\epsilon]$, see Figure \ref{boxes}, and $A_n$ be the event 
     $$A_n:=\left\{N_{n,\epsilon}\geq\tfrac{1}{M}\,\Big(|\partial_E\Lambda_{n}|+\sqrt{|\Lambda_{n}|}\Big)\right\}.$$
     The number on every single edge is a Poisson distributed random variable with parameter $\epsilon$, consequently
     having mean and variance $\epsilon$.
     
     As those random variables are independent, a choice of $\epsilon$ such that $\epsilon\leq\tfrac1M$ yields using
     Chebyshev's inequality:
     $$\Prob(A_n)\leq\Prob\Big(N_{n,\epsilon}-\E N_{n,\epsilon}\geq\tfrac{1}{M}\,\sqrt{|\Lambda_{n}|}\Big)
     \leq M^2\,\frac{\text{var}(N_{n,\epsilon})}{|\Lambda_{n}|}\leq M\,\frac{|\partial_E\Lambda_{n}|}{|\Lambda_{n}|}.$$
     
   \begin{figure}[H]
     \centering
     \includegraphics[scale=0.9]{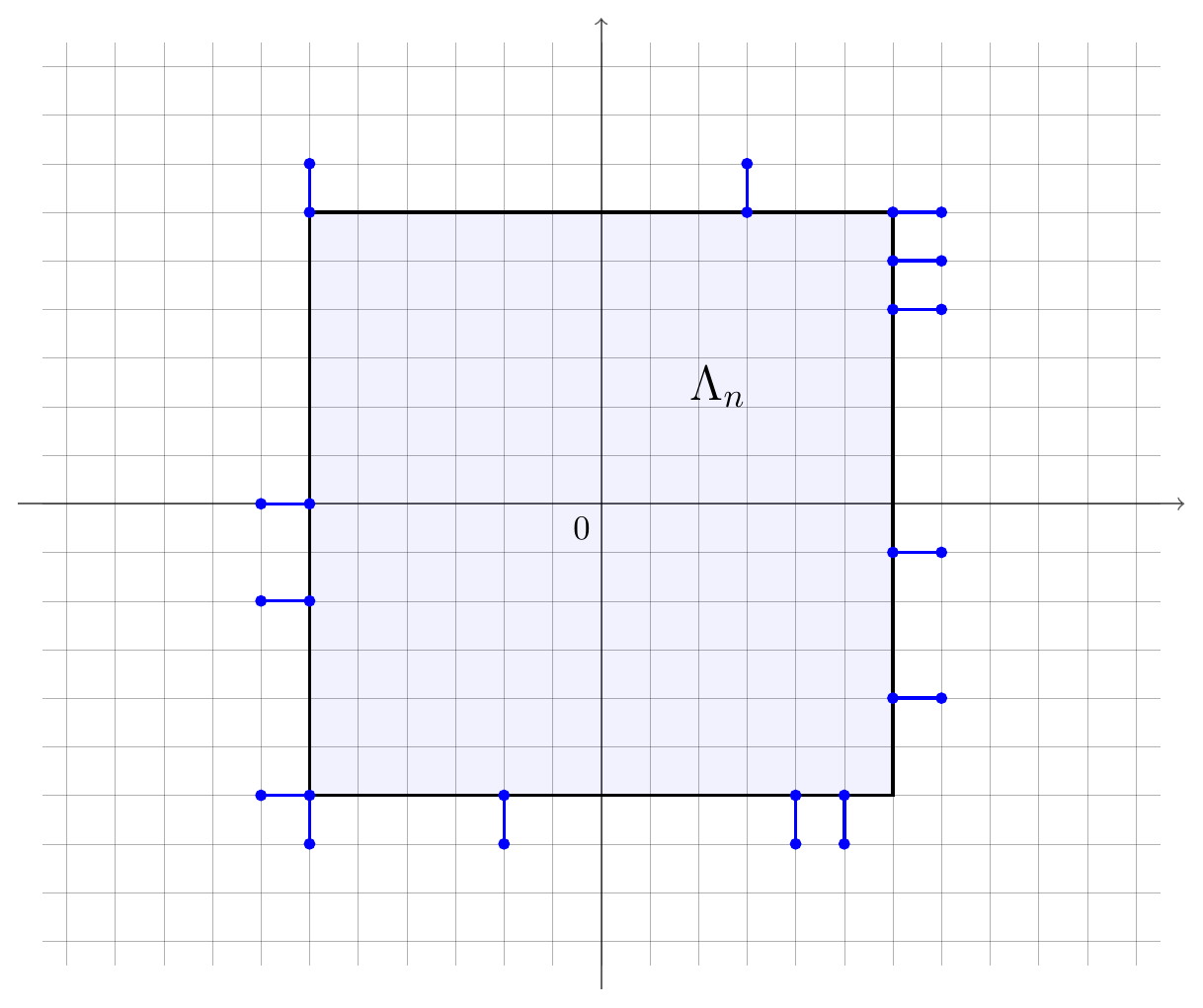}
     \caption{The interactions on the boundary of the box $\Lambda_{n}$ in the time interval\\
              $[t,t+\epsilon]$ are few compared to the size of the box for large $n$.\label{boxes}}
    \end{figure}
    \noindent     
     In view of (\ref{summable}), the Borel-Cantelli-Lemma shows that almost surely only finitely many
     $A_{n_k}$ will occur. In order to conclude, we have to show that this implies
     \begin{equation}\label{lts}
     \lim_{k\to\infty}\frac{1}{|\Lambda_{n_k}|}\sum_{v\in\Lambda_{n_k}}W^\text{tot}_{t+\epsilon}(v)=
     \lim_{k\to\infty}\frac{1}{|\Lambda_{n_k}|}\sum_{v\in\Lambda_{n_k}}W^\text{tot}_t(v),
     \end{equation}
     which in turn guarantees that the limit in (\ref{subseq}) is constant over time.
     
     It is not hard to convince yourself that Poisson events off $\partial_E\Lambda_{n_k}$ will not change
     $\sum_{v\in\Lambda_{n_k}}W^\text{tot}_{t}(v)$ and every single event on $\partial_E\Lambda_{n_k}$
     can change the sum of total energies in $\Lambda_{n_k}$ by at most $M$. Therefore, on the complement of
     $A_{n_k}$, we get that
     $$\frac{1}{|\Lambda_{n_k}|}\;\Bigg|\sum_{v\in\Lambda_{n_k}}W^\text{tot}_{t+\epsilon}(v)-
     \sum_{v\in\Lambda_{n_k}}W^\text{tot}_t(v)\Bigg|\leq \frac{M}{|\Lambda_{n_k}|}\cdot N_{n_k,\epsilon}
     <\frac{|\partial_E\Lambda_{n_k}|}{|\Lambda_{n_k}|}+\frac{1}{\sqrt{|\Lambda_{n_k}|}}.$$
     As this converges to $0$ when $k\to\infty$, we have shown that (\ref{lts}) holds almost surely, which
     concludes the proof.

\end{proof}

\begin{lemma}\label{asbeh}
  For the Deffuant model on the lattice $\Z^d$ as above, with threshold parameter $\theta\in(0,b]$,
  the following holds a.s.\ for every two neighbors $u,v\in\Z^d$:
  \begin{align}
   \begin{split}
    &\text{Either }|\eta_t(u)-\eta_t(v)|>\theta\text{ for all sufficiently large t, i.e. the edge }
    \langle u,v\rangle\\
    &\text{is finally blocked, or}\\
    &\lim_{t\to\infty}|\eta_t(u)-\eta_t(v)|=0,\text{ i.e. the two neighbors will finally concur.}
   \end{split}
  \end{align}
\end{lemma}

\begin{proof}
  The above lemma corresponds to Prop.\ 6.1 in \cite{ShareDrink} and the original proof
  generalizes to the higher-dimensional setting with only minor changes. 
  
  As the times between Poisson events on a single edge are exponentially distributed, the
  memoryless property ensures that given a finite collection of edges and some fixed time $s$, the
  edge which experiences the next Poisson event is chosen uniformly at random. Let us take
  $\En: x\mapsto x^2$ as energy function and fix $e=\langle u,v\rangle$ as well as
  some $\delta>0$. If there is a Poisson event at $e$ at time $t$ and the opinion values
  of $u$ and $v$ are not more than $\theta$ apart from each other, energy to the amount of
  $w_t(e)=2\mu\,(1-\mu)(\eta_{t-}(u)-\eta_{t-}(v))^2$ is lost along the edge, see (\ref{loss}).
  If $|\eta_{t-}(u)-\eta_{t-}(v)|\in(\delta,\theta]$, such an increase of $W_t^\text{loss}(e)$ would be at least
  $2\mu\,(1-\mu)\,\delta^2$. The opinion values of $u$ and $v$ can only change if one
  of the $4d-1$ edges incident to either $u$ or $v$ experiences a Poisson event. 
  Given $|\eta_{s}(u)-\eta_{s}(v)|\in(\delta,\theta]$ for some fixed time $s$, the probability that it is in
  fact $e$ where the first Poisson event after time $s$ on an edge incident to either $u$ or $v$ occurs
  is $\tfrac{1}{4d-1}$. 
  
  By the extended version of the Borel-Cantelli-Lemma (involving conditional probabilities,
  see e.g.\ Cor.\ 6.20 in \cite{Foundations}) such an increase will happen infinitely often,
  if $|\eta_t(u)-\eta_t(v)|\in(\delta,\theta]$ for arbitrarily large $t$,
  forcing $(W_t^\text{loss}(e))_{t\geq0}$ to diverge. This cannot happen with positive probability, since
  according to Lemma \ref{avpre} we have
  $\E[W_t^\text{loss}(e)]\leq2\, \E[W^\text{tot}_t(v)]=2\, \E[W_0(\mathbf{0})]\leq 2\,b^2$.
  Hence, it follows that a.s.\ $|\eta_t(u)-\eta_t(v)|\notin(\delta,\theta]$ for sufficiently
  large $t$. 
  
  For small values of $\delta$, more precisely $\delta<\tfrac\theta2$, the margin 
  $|\eta_t(u)-\eta_t(v)|$ cannot jump back and forth between $[0,\delta]$ and $(\theta,b]$,
  since single updates can change the value at any site by no more than 
  $\mu\theta\leq\tfrac\theta2$. Consequently, for $0<\delta<\tfrac\theta2$, the following
  holds almost surely:
  $$\limsup_{t\to\infty}|\eta_t(u)-\eta_t(v)|\in[0,\delta]\quad\text{or}\quad
    \liminf_{t\to\infty}|\eta_t(u)-\eta_t(v)|\in(\theta,b].$$   
  For $\delta$ can be chosen arbitrary small and there are only countably many edges,
  the claim is established.
\end{proof}\vspace*{1em}

\begin{lemma}\label{0-1}
The probability that there will be finally blocked edges is either $0$ or $1$.
\end{lemma}
\begin{proof}
Fix an edge $e=\langle u,v\rangle$ and assume that $\Prob(e\text{ is finally blocked})=0$.
By translation invariance of the process, this has to be true for all edges $e\in E$.
The union bound together with the preceeding lemma gives:
$$\Prob(\lim_{t\to\infty}|\eta_t(u)-\eta_t(v)|=0\ \forall\, u,v\in\Z^d)=1.$$
For $\Prob(e\text{ is finally blocked})>0$, let $N(v)$ denotes the number of edges incident to
site $v$ that are finally blocked. Then the ergodicity of $\{\eta_0(v)\}_{v\in\Z^d}$ and the
independent Poisson processes attributed to the edges with respect to shifts, forces that almost
surely the following holds (using Zygmund's Ergodic Theorem):
\begin{equation*}
      \lim_{n\to\infty}\frac{1}{|\Lambda_n|}\sum_{v\in \Lambda_n}N(v)
      =\E[N(\mathbf 0)]=2d\cdot\Prob(e\text{ is finally blocked})>0.
     \end{equation*}
Hence, with probability 1 infinitely many edges will be finally blocked.
\end{proof}\vspace*{1em}

\noindent
Having derived these auxiliary results, we can proceed to prove the main result of
this section:\vspace*{1em}

\begin{nproof}{of Theorem \ref{Zd}}
\begin{enumerate}[(a)]
\item
Given some confidence bound $\theta\geq\tfrac12$, the value at every vertex which is incident
to a finally blocked edge must be finally located in $[0,1-\theta)\cup(\theta,1]$. 
Due to Lemma \ref{asbeh} this holds for every vertex almost surely if there are edges
which are finally blocked. The foregoing lemma tells us, that if an edge is
finally blocked with positive probability, we get
\begin{equation}\label{liminf}
\liminf_{t\to\infty}|\eta_t(v)-\tfrac12|\geq \theta-\tfrac12\text{ for all }v\in\Z^d\text{ a.s.}
\end{equation}
Choosing the energy function $\En: x\mapsto|x-\tfrac12|$ and applying Lemma \ref{avpre}
we find:
   \begin{align*}
    \E\big[\liminf_{t\to\infty}W_t(v)\big]&=\E\big[\liminf_{t\to\infty}|\eta_t(v)-\tfrac12|\big]\\
    &\leq\liminf_{t\to\infty}\E\big[|\eta_t(v)-\tfrac12|\big]\\
    &\leq\liminf_{t\to\infty} \E[W^\text{tot}_{t}(v)]\\&= \E[W^\text{tot}_{0}(v)]=\tfrac14,
   \end{align*}
where Fatou's Lemma was used in the first inequality and the non-negativity of $W_t^\text{loss}(e)$ in the second.
If we assume $\Prob(e\text{ is finally blocked})>0$ for some, hence any $e$, the first expectation
must be at least $\theta-\tfrac12$ by (\ref{liminf}), which leads to a contradiction if $\theta$ is larger than $\tfrac34$.
\item Note that no special feature of $\text{\upshape unif}([0,1])$ was used, but 
   $\E\big[|\eta_0-\tfrac12|\big]=\tfrac14$.
   Consequently, the above result still holds if $\text{\upshape unif}([0,1])$ is replaced by 
   some other distribution $\mathcal{L}(\eta_0)$ on $[0,1]$ and the bound $\tfrac34$ replaced by 
   $\E\big[|\eta_0-\tfrac 12|\big]+\tfrac 12$ simultaneously. Furthermore, this bound is non-trivial, 
   i.e.\ less than 1, provided $\Prob(\eta_0\in\{0,1\})<1$ for this implies
   $\E\big[|\eta_0-\tfrac 12|\big]<\tfrac12$. If however $\eta_0\in\{0,1\}$ almost surely, trivially only
   $\theta=1$ will not allow for finally blocked edges, given $\eta_0$ is not a.s.\ constant.
\end{enumerate}\vspace*{-1em}
\end{nproof}\vspace*{-1em}

\begin{remark}
\begin{enumerate}[(a)]
\item 
   There are two major differences to the results on $\Z$. Firstly, even if intuitively appealing it is no
   longer ensured that weak consensus as described in Theorem \ref{Zd} will lead to consensus in the strong sense, i.e.\
   that every individual value converges to the mean. By ergodicity we know
   \begin{equation*}
      \lim_{n\to\infty}\frac{1}{|\Lambda_n|}\sum_{v\in \Lambda_n}\mathbbm{1}_{\{\lim_{t\to\infty}\limits\eta_t(v) 
      \text{ exists}\}}=\Prob\big(\lim_{t\to\infty}\eta_t({\mathbf 0})\text{ exists}\big).
     \end{equation*}
   In the case of consensus, the indicator functions on the left hand side are either all $0$ or all $1$.
   In other words, for $\theta$ such that weak consensus is guaranteed, the existence of the limits is an
   event with probability either $0$ or $1$. In the latter case another application of ergodicity and dominated
   convergence show that this limit must be the mean of the initial distribution:
   \begin{align*}
      \lim_{t\to\infty}\eta_t(v)&=\lim_{n\to\infty}\frac{1}{|\Lambda_n|}
      \sum_{u\in \Lambda_n}\lim_{t\to\infty}\eta_t(u)\\
      &=\E\big[\lim_{t\to\infty}\eta_t(v)\big]\\
      &=\lim_{t\to\infty}\E[\eta_t(v)]=\E[\eta_0(v)],
     \end{align*}
   where the first equality follows from weak consensus, the last is Lemma \ref{avpre} with the identity
   as energy function.
   
   Secondly, it is no longer clear that we can talk about a critical value for $\theta$ separating the parameter
   space neatly into a sub- and a supercritical regime, since final consensus is not
   necessarily an increasing event in $\theta$. By Lemma \ref{0-1} it is clear that for fixed $\theta$ we have
   that all neighbors finally concur with probability either $0$ or $1$. Hence both cases can not occur
   simultaneously but there might be a range for $\theta$ in which they alternate, unlike in the case of $\Z$.
\item 
   Let us next consider another example. Taking for instance $\text{\upshape unif}(\{0,\tfrac12,1\})$
   as distribution of the initial values,
   the reasoning in part (b) of the theorem shows  that finally blocked edges are in this case only possible for
   $$\theta\leq\E\big[|\eta_0-\tfrac 12|\big]+\tfrac12=\tfrac13+\tfrac12=\tfrac56.$$ 
   For other distributions it might even be beneficial to choose some different convex energy 
   function giving a potentially sharper bound on $\theta\geq\tfrac 12$ of the kind:
   The probability for finally blocked edges can only be non-zero for $\theta$ such that
   \begin{equation*}
   \inf\{\En(x),\; x\in[0,1-\theta)\cup(\theta,1]\}\leq\E\big[\En(\eta_0)\big].
   \end{equation*}
   Clearly, this inequality is trivial if the minimal value $\min\{\En(x),\; x\in[0,1]\}$
   is attained on $[0,1-\theta)\cup(\theta,1]$. If this is not the case, it reads 
   \begin{equation}\label{bound}
    \min\{\En(1-\theta),\En(\theta)\}\leq\E\big[\En(\eta_0)\big],
   \end{equation}
   due to the convexity of $\En$. Choosing $\En$ such that it vanishes on the
   support of $\mathcal{L}(\eta_0)$ will only give the trivial bound $\theta\leq\tfrac12+\sup
   \{|x-\tfrac12|,\; x\in \supp(\mathcal{L}(\eta_0))\}$.
   
   In addition, Jensen's inequality tells us that regardless of the chosen convex energy
   function, from (\ref{bound}) we cannot get a bound on $\theta$ so sharp that
   $\E\eta_0\notin(1-\theta,\theta)$. Since in this case we trivially have 
    \begin{equation*}
   \inf\{\En(x),\; x\in[0,1-\theta)\cup(\theta,1]\}\leq\En\big(\E\eta_0\big)
   \leq\E\big[\En(\eta_0)\big].
   \end{equation*}
   Finally, a gap in the distribution of $\eta_0$ also reduces the scope of (\ref{bound}), since for
   $\Prob(\eta_0\in(1-\theta,\theta))=0$ we get: 
   $$\En(\eta_0)\geq\inf\{\En(x),\; x\in[0,1-\theta)\cup(\theta,1]\}\text{ a.s.}$$
   This trivially implies the above inequality.
   
   In summary, the same factors obstructing consensus in the Deffuant model on $\Z$ reappear in this
   treatment of the higher-dimensional case (cf.\ part (a) of Theorem \ref{gen}).
   
\item Next, it is worth noting that the energy function chosen in the proof of Theorem \ref{Zd} is in fact
   best possible regarding (\ref{bound}) for symmetric distributions. If $\En$ is rescaled by some positive
   factor or translated by adding a constant, the inequality (\ref{bound}) stays unchanged. As the
   inequality is symmetric around $\tfrac12$ for symmetric distributions, it holds for the pair
   $(x\mapsto\En(x),\theta)$ if and only if it holds for $(x\mapsto\En(1-x),\theta)$. 
   A symmetrization of the kind $\tilde{\En}(x)=\tfrac12\,(\En(x)+\En(1-x))$
   will thus not change the right-hand side and at most increase the left-hand side if 
   $\En(\theta)\neq\En(1-\theta)$, making the condition only stricter.
   
   Therefore, an energy function giving the best bound on parameters $\theta$ allowing for finally
   blocked edges through (\ref{bound}) can be assumed to be symmetric on $[0,1]$ and having the image
   set $[0,\tfrac12]$. Set $X:=\tfrac12+|\eta_0-\tfrac12|$, a $[\tfrac12,1]$-valued random variable, 
   which by the symmetry of $\eta_0$ implies $\E[\En(X)]=\E[\En(\eta_0)]$. 
   The largest $\theta$ satisfying (\ref{bound}) is then the unique one (larger than $\tfrac12$) for which
   $\En(\theta)=\E[\En(\eta_0)]$. Note that the convexity of the energy function
   forces it to be strictly monotonous where it is not attaining its minimum, which is 0, and a
   choice such that $\En(\eta_0)=0$ a.s.\ will only give a trivial bound on $\theta$ as discussed
   above.
   
   Another look at Jensen's inequality tells us that $\E[\En(X)]\geq
   \En(\E X)$, with strict inequality if $\En$ is not linear on $\supp(\mathcal{L}(X))$.
   If this inequality is strict, larger values for $\theta$ than $\E X$ will also satisfy (\ref{bound}).
   Being linear on $\supp(\mathcal{L}(X))$ and convex means being linear at least on the smallest interval 
   containing the support, i.e.\ $I:=\text{conv}(\supp(\mathcal{L}(X)))$. How $\En$ is defined
   on $[\tfrac12,1]\setminus I$ is irrelevant, so we may assume it to be linear on all of $[\tfrac12,1]$.
   The assumptions on symmetry and image set finally force $\En$ to be the function $x\mapsto|x-\tfrac12|$.
 \vspace{0.3cm}
 \begin{figure}[H]
     \hspace*{8.0cm}
     \unitlength=0.80mm
     \begin{picture}(50.00,37.00)
          {\color[rgb]{.4,.4,.4}
          \put(-2.00,0.00){\vector(1,0){51.00}}
          \put(-1.50,10.00){\line(1,0){3.00}}
          \put(-1.50,20.00){\line(1,0){3.00}}
          \put(0.00,-2.00){\vector(0,1){27.00}}
          \put(30.00,-1.50){\line(0,1){3.00}}
          \put(45.00,-1.50){\line(0,1){3.00}}
          \put(-4.30,-5.00){$\scriptstyle 0$}
          \put(-5.00,18.90){$\scriptscriptstyle\tfrac12$}
          \put(-5.00,8.90){$\scriptscriptstyle\tfrac14$}
          \put(28.70,-6.00){$\scriptscriptstyle\tfrac23$}
          \put(44.20,-5.00){$\scriptstyle 1$}}

          \put(0.00,10.00){\line(3,-1){30.00}}
          \put(30.00,0.00){\line(3,4){15.00}}
          \put(37.00,15.00){$\En$}
          {\color[rgb]{0,0,1} \linethickness{2pt}
          \put(5.00,12.00){$\scriptstyle\mathcal{L}(\eta_0)$}
          \put(0.00,0.00){\line(0,1){13.33}}
          \put(30.00,0.00){\line(0,1){20.00}}
          \put(45.00,0.00){\line(0,1){6.66}}}
     \end{picture}
  \end{figure}\vspace{-3.85cm}
  
\item In the case of an asymmetric distribution of $\eta_0$ there are actually better choices.
   \par
   \begingroup
   \rightskip14em
   Consider the example sketched on the right, where
   $\Prob(\eta_0=0)=\tfrac13,\ \Prob(\eta_0=\tfrac23)=\tfrac12$,
   $\Prob(\eta_0=1)=\tfrac16$, and the energy function is piecewise linear as shown.

   Taking $x\mapsto|x-\tfrac12|$ as energy function shows via (\ref{bound}) that finally
   blocked edges are only possible for \par\endgroup
   \vspace{-1em}
   $$\theta\leq\E\big[|\eta_0-\tfrac 12|\big]+\tfrac12=\tfrac12\,(\tfrac12+\tfrac16)+\tfrac12=\tfrac56.$$
   Taking $\En$ piecewise linear with $\En(0)=\tfrac14,\ \En(\tfrac23)=0\text{ and }\En(1)=\tfrac12$ gives
   in turn $\E[\En(\eta_0)]=\tfrac16=\En(\tfrac29)=\En(\tfrac79)$, hence a.s.\ no blocked edges
   for $\theta>\tfrac79$, which is slightly better.

   Note however that for every convex $\En$ there are always linear functions $l_1,\ l_2$ such that
   $l_1(1-\theta)=\En(1-\theta),\ l_2(\theta)=\En(\theta)$ and $l_1,l_2\leq\En$. Taking their maximum
   will give a convex function leaving the left-hand side of (\ref{bound}) unchanged and at most decreasing
   the right-hand side. By an appropriate affine transformation $y\mapsto a\,y+c,\ a>0$ this function
   can be altered to have image set $[0,\tfrac12]$ without changing the condition on $\theta$ that follows
   from (\ref{bound}) as mentioned above.
   Consequently, the sharpest bound using (\ref{bound}) will even in the asymmetric case always be
   established by some piecewise linear function with only one bend mapping to $[0,\tfrac12]$ as in
   the example.

\item It is worth remarking, that the bounds coming from (\ref{bound}) applied to the model with i.i.d.
   initial opinions on $\Z$ are a lot closer to the truth for centered distributions.

   The best we can come up with for the uniform case is $\tfrac34$ and for $\text{\upshape unif}(\{0,\tfrac12,1\})$
   even $\tfrac56$, whereas Theorem \ref{gen} tells us that on $\Z$ the actual bound on $\theta$ to
   allow for finally blocked edges is $\tfrac12$ in either case. In the asymmetric example from above, we get the
   bound $\theta\leq\tfrac79$ which is not too far off its critical value $\theta_\text{c}=\tfrac23$ on $\Z$. 

   For a distribution of $\eta_0$ which is really concentrated around the mean, e.g.\ 
   $\Prob(\eta_0=0)=\Prob(\eta_0=1)=\tfrac1n,\ \Prob(\eta_0=\tfrac12)=\tfrac{n-2}{n}$, with $n$ large, the bound
   derived using $x\mapsto|x-\tfrac12|$ as energy function is 
   $\theta\leq\E\big[|\eta_0-\tfrac 12|\big]+\tfrac12=\tfrac1n+\tfrac12$. The corresponding critical value 
   on $\Z$ according to Theorem \ref{gen} is again $\tfrac12$, hence quite well approximated.

   That we get the right answer for a non-constant distribution concentrated on $\{0,1\}$ is due to
   the huge gap. For a slightly changed symmetric version, i.e.\ 
   $\Prob(\eta_0=0)=\Prob(\eta_0=1)=\tfrac{n-1}{2n},\ \Prob(\eta_0=\tfrac12)=\tfrac1n$, again $n$ large,
   however, the best bound we get following the reasoning of the above theorem is
   $$\theta\leq\E\big[|\eta_0-\tfrac 12|\big]+\tfrac12=\tfrac12\cdot\tfrac{n-1}{n}+\tfrac12=1-\tfrac{1}{2n}$$
   and this is far off the true value on $\Z$, which is once more $\theta_\text{c}=\tfrac12$.

\item As in Theorem \ref{gen}, the general case where the initial distribution's support is contained in
   $[a,b],\ a<b\in\R$, can be treated by appropriate translation and scaling.
\end{enumerate}
\end{remark}
\noindent
     In conclusion, the results from Section \ref{sec2} show that for $d=1$ and a sequence of initial
     values satisfying the finite energy condition (see Definition \ref{finen}), there exists a critical parameter
     $\theta_\text{c}$ (which is $\tfrac12$ in the standard uniform case) at which a phase transition
     from no consensus to strong consensus takes place. Strictly weak consensus could only exist for the
     unsolved case of $\theta=\theta_\text{c}$.

     Theorem \ref{Zd} states that the case of no consensus is impossible for initial marginal distributions
     that attribute a positive probability to $(0,1)$ and $\theta$ large enough ($\tfrac34$ in the
     uniform case).

\begin{remark}\label{amenable}
   The results from Theorem \ref{Zd} can actually be generalized from the grid $\Z^d$ to any infinite,
   locally finite, transitive and amenable (connected) graph $G=(V,E)$. In this generality, the configuration
   of initial opinions would have to be ergodic with respect to the graph automorphisms instead of shifts,
   of course.
   
   Recall that a graph is called locally finite if every vertex has a finite degree, which together with the
   regularity of a transitive graph implies bounded degree. A graph is called {\em amenable} if there exists a
   sequence $(F_n)_{n\in\N}$ of finite sets such that the ratio of boundary and volume $\tfrac{|\partial_E F_n|}{|F_n|}$
   tends to $0$ as $n\to\infty$. Such sequences are called {\em F\o lner sequences}. 
   
   In the case of an infinite, locally finite, transitive and amenable connected graph, we can choose the
   F\o lner sequence $(F_n)_{n\in\N}$ as an increasing set sequence with $\bigcup_{n\in\N}F_n=V$; see the
   appendix of \cite{amenable} for further details. As a replacement for Zygmund's ergodic theorem, we can
   then use the mean ergodic theorem for $L^2$-functions which can be found as Thm.\ A.5 in \cite{amenable},
   with $(F_n)_{n\in\N}$ stepping in for $(\Lambda_n)_{n\in\N}$:
   $$\lim_{n\to\infty}\frac{1}{|F_n|}\sum_{v\in F_n}W^\text{tot}_t(v)=\E[W^\text{tot}_t(\mathbf{0})]\quad\text{in }L^2,$$
   where $\mathbf{0}$ is some fixed vertex of $G$.
   It is not a problem that this result only gives $L^2$-convergence instead of almost sure convergence,
   since $L^2$-convergence is stronger than convergence in probability and the latter implies almost sure convergence
   of a subsequence, which is enough for our purposes.
\end{remark}    

\subsection{Consequences in terms of stochastic dominance}

From the area of probabilistic risk analysis the following orders of stochastic dominance are known,
which make it possible to rewrite the results from the foregoing subsection obtained by using energy arguments
in a nice way.
\begin{definition}
Let $X,Y$ be two random variables with finite expectation and $\mathcal{F}_{\text{cx}}$ 
denote the set of all convex, $\mathcal{F}_{\text{icx}}$ the set of all increasing convex functions on $\R$.
\begin{enumerate}[(i)]
\item $X$ is said to be {\itshape smaller than} $Y$ {\itshape in the usual stochastic order}, commonly denoted by 
 $X\leq_{\text{st}}Y$, if for all $a\in\R$:
 $$\Prob(X>a)\leq\Prob(Y>a).$$
\item $X$ is said to be {\itshape smaller than} $Y$ {\itshape in the convex order}, commonly denoted by 
 $X\leq_{\text{cx}}Y$, if for all functions $\phi\in\mathcal{F}_{\text{cx}}$
 for which the corresponding expectations exist:
 $$\E[\phi(X)]\leq\E[\phi(Y)].$$
\item $X$ is said to be {\itshape smaller than} $Y$ {\itshape in the increasing
 convex order}, commonly denoted by $X\leq_{\text{icx}}Y$, if for all functions $\phi\in\mathcal{F}_{\text{icx}}$
 for which the corresponding expectations exist:
 $$\E[\phi(X)]\leq\E[\phi(Y)].$$
\end{enumerate}
\end{definition}
\noindent
It is obvious from the definition that $\leq_{\text{\upshape cx}}$ implies $\leq_{\text{\upshape icx}}$.
Furthermore, the converse is true, if the expectations of both random variables coincide, i.e.
$$X\leq_{\text{cx}}Y\ \Leftrightarrow\ X\leq_{\text{icx}}Y\text{ and }\E X=\E Y,$$
see for example Thm.\ 4.A.35 in \cite{Orders}.

An easy coupling argument (using quantile transformation) shows that $\leq_{\text{\upshape st}}$ implies $\leq_{\text{\upshape icx}}$.

\begin{proposition}
Let $(\eta_t(v))_{t\geq 0}$ denote the piecewise constant jump process describing the value at some fixed
vertex $v\in\Z^d$ throughout time, as before. Furthermore, let the initial values again be distributed
on $[0,b]$ and $\E\eta_0$ be the corresponding expected value.

For any two points in time $0\leq s\leq t$, we have $\eta_t(v)\leq_{\text{\upshape cx}}\eta_s(v)$.
This in turn directly implies $|\eta_t(v)-\E\eta_0|\leq_{\text{\upshape icx}}|\eta_s(v)-\E\eta_0|$.
\end{proposition}

\begin{proof}
First of all, it is worth remarking that the partial orders $\leq_{\text{cx}}$ and $\leq_{\text{icx}}$ are
actually defined on the set of distributions and do therefore not depend on a random variable $X$ itself but
rather on $\mathcal{L}(X)$. The distribution of $\eta_t(v)$ is by symmetry the same for every $v\in\Z^d$, hence
it is enough to consider one fixed vertex.

Let $\phi$ be a convex function on $\R$. For every $t\geq0$ the random variable $\eta_t(v)$ lies in $[0,b]$ 
and since convexity implies continuity on closed intervals, $\phi$ attains its minimum
$$c:=\min\big\{\phi(x),\;x\in[0,b]\big\}.$$ 
Hence $\En:x\mapsto\phi(x)-c$ is a non-negative convex function on $[0,b]$ and therefore a proper choice as
energy function as outlined in the beginning of the foregoing subsection.

Let $W_t(v)=\En(\eta_t(v))$ denote the energy attributed to the chosen vertex at time $t$ and
$W^\text{tot}_t(v)=W_t(v)+\frac12\sum_{e\in E(v)}W_t^\text{loss}(e)$ its total energy, just as in (\ref{tot}).
Lemma \ref{avpre} tells us that $\E[W^\text{tot}_t(v)]=\E[W_0(v)]$ for all $t\geq0$ and the fact that
$(W_t^\text{loss}(e))_{t\geq0}$ is non-decreasing and non-negative for every edge $e$ gives accordingly 
\begin{align*}\E[W_t(v)]&=\E[W^\text{tot}_t(v)]-\frac12\sum_{e\in E(v)}\E[W_t^\text{loss}(e)]\\
&\leq\E[W^\text{tot}_s(v)]-\frac12\sum_{e\in E(v)}\E[W_s^\text{loss}(e)]=\E[W_s(v)]\\
&\leq\E[W_0(v)]\quad \text{for }0\leq s\leq t.
\end{align*}
If we plug in the special form of $\En$ chosen above (and add $c$ along the chain of inequalities) this reads:
$$\E\big[\phi(\eta_t(v))\big]\leq\E\big[\phi(\eta_s(v))\big]
\quad\Big(\,\leq\E\big[\phi(\eta_0(v))\big]\,\Big).$$
Since $\phi\in\mathcal{F}_{\text{cx}}$ was arbitrary, this proves the first part of the claim.

To see that $(|\eta_t(v)-\E\eta_0|)_{t\geq 0}$ is a non-increasing sequence with respect to $\leq_{\text{\upshape icx}}$
one only has to note that the function $x\mapsto|x-\E\eta_0|$ is convex. A short moment's thought reveals that
the composition of an increasing convex with a convex function is again convex. Thus, for 
$\phi\in\mathcal{F}_{\text{icx}}$ the already proved part applied to the function $x\mapsto\phi(|x-\E\eta_0|)$ provides
$$\E\big[\phi(|\eta_t(v)-\E\eta_0|)\big]\leq\E\big[\phi(|\eta_s(v)-\E\eta_0|)\big],$$
which in turn proves $|\eta_t(v)-\E\eta_0|\leq_{\text{icx}}|\eta_s(v)-\E\eta_0|$.
\end{proof}\vspace*{1em}

\noindent
This proposition in hand makes it possible to reprove the result from Theorem \ref{Zd}:
Already in 1979, Meilijson and Nádas \cite{olderthanMori} showed that $Y\leq_{\text{\upshape icx}}X$ implies
$Y\leq_{\text{\upshape st}}h_{\mathcal{L}(X)}(X)$, where the function $h_\mu$ denotes the mean residual
life of a random variable with distribution $\mu$, i.e.:
$$\text{For }Z\sim\mu \text{ and } t\in\R \text{ s.t. }\mu\big((t,\infty)\big)>0:\ h_\mu(t):=\E[Z\,|\,Z>t].$$
Having the initial distribution $\mathcal{L}(\eta_0)=\text{unif}([0,1])$ means 
$|\eta_0-\tfrac12|\sim\text{unif}([0,\tfrac12])$, which gives
$$h_{\text{unif}([0,{^1\!/\,\!_2}])}(t)=\tfrac14+\tfrac{t}{2}.$$
Consequently, we get $|\eta_t-\tfrac12|\leq_{\text{\upshape st}}\tfrac14+\tfrac{Z}{2}$, where
$Z\sim\text{unif}([0,\tfrac12])$, another contradiction to (\ref{liminf}) if $\theta>\tfrac34$.\\[1em]
\noindent
That the processes $(\eta_t(v))_{t\geq0}$ are non-increasing in the convex order renders it possible to
conclude convergence in distribution. This however is far from the almost sure convergence derived in
the one-dimensional case.

\begin{proposition}
Let $(\eta_t(v))_{t\geq 0}$ be as before. There exists a $[0,b]$-valued random variable $\eta_\infty$
such that $\eta_t(v)\law\eta_\infty$ for every $v\in\Z^d$.
\end{proposition}

\begin{proof}
Again, symmetry ensures that if the statement holds true for some vertex $v$ it is valid for all such.
Building on a famous result of Stra\ss en and following ideas of Doob, Kellerer showed in 1972 that
for a family of probability measures $\{\mu_t\}_{t\geq0}$ which is non-decreasing in the increasing convex 
order there always exists a submartingale with the corresponding marginals, see Thm.\ 3 in \cite{Kellerer}.
Therefore, the non-increasing family $\{\mathcal{L}(\eta_t(v))\}_{t\geq0}$ can be interpreted as the marginal
distributions of a supermartingale $(X_t)_{t\geq0}$.
As the mean of these distributions is constant, which follows from Lemma \ref{avpre} as mentioned in the
above remark and corresponds to the stronger condition of non-increasing ordering w.r.t.\ $\leq_{\text{cx}}$,
$(X_t)_{t\geq0}$ actually is a martingale.

Doob's martingale convergence theorem guarantees a random
variable $X_\infty$ such that $(X_t)_{t\geq0}$ converges to $X_\infty$ almost surely, hence in distribution.
Writing $\eta_\infty$ instead of $X_\infty$ establishes the claim.
\end{proof}

\section{On the infinite cluster of supercritical bond percolation}

In this section we consider the Deffuant opinion dynamics on the random subgraph
of $\Z^d$, $d\geq2$, which is formed by supercritical i.i.d.\ bond percolation, independent of the initial
configuration and the Poisson processes determining the times of potential opinion updates. 

That means, each edge of the grid is independently chosen to be open with a fixed probability $p\in(0,1]$.
One of the classical results in percolation theory tells us that for $d\geq2$, there exists a critical
value $p_c(d)\in(0,1)$ for $p$ above which we will a.s.\ find an infinite cluster and that this cluster is a.s.\ 
unique.
The common notation for the event that some vertex $v$ sits in the infinite cluster is $\{v\leftrightarrow
\infty\}$. Slightly abusing this notation we will write $\{e\leftrightarrow\infty\}$ for the event that
the edge $e$ is part of the infinite cluster.

The fact that ergodicity, one essential element to derive the results from the foregoing section,
is preserved when we consider the (random) subgraph of $\Z^d$ formed by i.i.d.\ bond percolation
allows for an immediate transfer of the corresponding results for the whole grid.

\begin{lemma}\label{perc}
Let the Deffuant model with initial values drawn from a distribution on $[0,b]$ and
parameter $\theta\in(0,b]$ be as above, but now take place on the graph of a supercritical i.i.d.\
bond percolation on $\Z^d$ which is independent of the initial configuration and the Poisson processes.
Then the lemmas of the foregoing section extend as follows:
\begin{enumerate}[(a)]
\item $\E[W_t^{\text{\upshape tot}}(v)\,|\,v\leftrightarrow\infty]=\E[W_0(\mathbf{0})]$
\item Given the edge $\langle u,v\rangle$ is open, we get as in Lemma \ref{asbeh} that a.s.\\
$|\eta_t(u)-\eta_t(v)|>\theta$ for sufficiently large $t$ or $\lim_{t\to\infty}|\eta_t(u)-\eta_t(v)|=0$.
\item The probability that some edges of the infinite cluster will be finally blocked in the Deffuant
model is either $0$ or $1$.
\end{enumerate}
\end{lemma}

\begin{proof}
\begin{enumerate}[(a)]
\item Using the notation from Lemma \ref{avpre} and its line of reasoning, it is obvious that
  the process $\{W^\text{tot}_t(v)\cdot\mathbbm{1}_{\{v\leftrightarrow\infty\}}\}_{v\in\Z^d}$ is ergodic
  with respect to shifts. Hence instead of (\ref{Birkhoff}) one has
  \begin{equation}
   \lim_{n\to\infty}\frac{1}{|\Lambda_n|}\sum_{v\in C_\infty\cap\Lambda_n}W^\text{tot}_t(v)
   =\E[W^\text{tot}_t(\mathbf{0})\cdot\mathbbm{1}_{\{\mathbf{0}\leftrightarrow\infty\}}] 
        \text{  a.s.,}
  \end{equation}
  where $C_\infty$ denotes the infinite percolation cluster. By the same argument as in the quoted lemma,
  the left-hand side is constant over time and we thus get
  \begin{eqnarray*}\Prob(v\leftrightarrow\infty)\cdot\E[W_t^{\text{\upshape tot}}(v)\,|\,v\leftrightarrow\infty]
   &=&\E[W_t^{\text{\upshape tot}}(v)\cdot\mathbbm{1}_{\{v\leftrightarrow\infty\}}]\\
   &=&\E[W_t^{\text{\upshape tot}}(\mathbf{0})\cdot\mathbbm{1}_{\{\mathbf{0}\leftrightarrow\infty\}}]\\
   &=&\E[W_0(\mathbf{0})\cdot\mathbbm{1}_{\{\mathbf{0}\leftrightarrow\infty\}}]\\
   &=&\Prob(\mathbf{0}\leftrightarrow\infty)\cdot\E[W_0(\mathbf{0})],
  \end{eqnarray*}
  using symmetry and independence. Dividing by the probability for percolation of a given vertex
  $\Prob(v\leftrightarrow\infty)$, which is non-zero for supercritical percolation, yields the claim.
\item To get the second statement one simply has to mimick Lemma \ref{asbeh}. The only things changing
  are that we have to condition on the event of $e=\langle u,v\rangle$ being open in the realization of the
  i.i.d.\ bond percolation and the probability at a given point in time that $e$ will be the next edge
  incident to either $u$ or $v$ where a Poisson event occurs is no longer precisely $\tfrac{1}{4d-1}$
  but bounded from below by the same value (since some of the other edges might be closed).
\item Following the proof of Lemma \ref{0-1}, let us consider the probability that some given edge $e$ is open,
  further belongs to the infinite percolation component and is finally blocked in the Deffuant dynamics. If
  $$p_{\text{\tiny block}}:=\Prob(e\leftrightarrow \infty, e\text{ finally blocked})=0,$$
  the union bound and part (b) guarantee that a.s.\ all neighbors in the infinite component will finally concur.
  If this probability is positive, however, and $N(v)$ denotes the number of edges incident to $v$, open in the
  realization of the i.i.d. bond percolation, that will get finally blocked in the Deffuant model, another
  application of Zygmund's Ergodic Theorem yields:
  \begin{equation*}
   \lim_{n\to\infty}\frac{1}{|\Lambda_n|}\sum_{v\in C_\infty\cap\Lambda_n}N(v)
   =\E\big[N(\mathbf 0)\cdot\mathbbm{1}_{\{\mathbf 0\leftrightarrow\infty\}}\big]
   =2d\cdot p_{\text{\tiny block}}>0\text{ a.s.}
  \end{equation*}
  Hence with probability 1, there will be (infinitely many) edges that belong to the infinite percolation 
  component and are finally blocked.
\end{enumerate}\vspace*{-0.1cm}
\end{proof}\vspace{-1em}\vspace*{0.1cm}

\noindent Having checked that these auxiliary results transfer appropriately to the setting of
supercritical percolation, the following equivalent to Theorem \ref{Zd} can be verified with the very same
reasoning as before:
\begin{theorem}\label{perc1}
Consider the Deffuant model on the subgraph of $\Z^d$, $d\geq2$, formed by an independent supercritical
i.i.d.\ bond percolation as described above.
\begin{enumerate}[(a)]
\item  If the initial values are distributed uniformly on $[0,1]$ and $\theta>\tfrac34$, a.s.\ 
       we will finally have weak consensus in the infinite percolation cluster, i.e.
       for all $u,v\in\Z^d$ given the event $\{u,v\leftrightarrow\infty\}$ we have
       $$\Prob\big(\lim_{t\to\infty}|\eta_t(u)-\eta_t(v)|=0\big)=1.$$       
\item  For general initial distributions on $[0,1]$, the range of $\theta$, where final consensus of
       the infinite cluster is guaranteed, is non-trivial, i.e.\ including values smaller than $1$,
       unless the initial values are concentrated on $0$ and $1$, taking on both values
       with positive probability.
\end{enumerate}
\end{theorem}

\begin{proof}
Given the event that $v\in\Z^d$ is in the infinite percolation cluster which contains (open) edges that
are finally blocked by the opinion dynamics we get as in (\ref{liminf})
\begin{equation*}
\liminf_{t\to\infty}|\eta_t(v)-\tfrac12|\geq \theta-\tfrac12\text{ a.s.}
\end{equation*}
Choosing again $\En: x\mapsto|x-\tfrac12|$ as energy function the above lemma and the conditional version
of Fatou's Lemma yield the following chain of inequalities:
   \begin{align*}
    \theta-\tfrac12&\leq\E\big[\liminf_{t\to\infty}|\eta_t(v)-\tfrac12|\;\big|\;v\leftrightarrow\infty\big]\\
    &\leq\liminf_{t\to\infty}\E\big[|\eta_t(v)-\tfrac12|\;\big|\;v\leftrightarrow\infty\big]\\
    &\leq\liminf_{t\to\infty} \E\big[W^\text{tot}_{t}(v)\;\big|\;v\leftrightarrow\infty\big]\\
    &= \E\big[W^\text{tot}_{0}(v)\big]=\E\big[|\eta_{0}(v)-\tfrac12|\big].
   \end{align*}
Consequently, for blocked edges to occur in the infinite percolation cluster we have to have
$\theta\leq\tfrac34$ in the standard case of {\upshape unif}$([0,1])$ initial opinion values and
 $\theta\leq\tfrac12+\E\big[|\eta_{0}(v)-\tfrac12|\big]$ in the general case.
\end{proof}\vspace*{1em}

\noindent
So far, this seems like just a generalization of Section 3. In the percolation setting
however, a coupling argument allows to prove a result concerning the other end of the $\theta$-spectrum,
under slightly stronger conditions on the initial opinion configuration (see also Remark \ref{weakercon} below).

\begin{theorem}\label{perc2}
Consider again the Deffuant model on the infinite cluster of supercritical percolation, this time with i.i.d.\
initial opinion values distributed on $[0,1]$, s.t. $[0,1]$ is the minimal closed interval containing the
support of the marginal distribution. In addition, we require the percolation parameter $p$ to be less
than 1.

For $\theta<\tfrac12$ the probability that the opinion dynamics approach strong consensus on the infinite
percolation cluster is 0. 
\end{theorem}

\begin{proof}
The line of reasoning to prove this statement is by contradiction. Assuming strong consensus for some fixed
value of $\theta$ in $(0,\tfrac12)$, we are going to show that there will be finally blocked edges in the
infinite percolation component with positive probability. This contradicts part (c) of Lemma \ref{perc}.

To that end let us consider two coupled copies of the supercritical i.i.d.\ bond percolation, see Figure
\ref{coupling}. Fix an edge
$e=\langle u,v\rangle$ and let the two copies coincide on $E(\Z^d)\setminus\{e\}$. Let $p\in(0,1)$ denote
the probability for an edge to be open in the percolation model and $A$ be the event that the edges incident
to $u$ other than $e$ are closed and $v$ sits in the infinite component. By a coupling argument
using local modification it can easily be seen that this event has positive probability if $p$ is supercritical.
\vspace{0.3cm}  
\begin{figure}[h]
     \centering
     \unitlength=0.5mm
     \begin{picture}(200.00,110.00)
          \put(0.00,20.00){\line(1,0){8.50}}
          \put(11.50,20.00){\line(1,0){17.00}}
          \put(51.50,20.00){\line(1,0){17.00}}
          \put(71.50,20.00){\line(1,0){8.50}}
          \put(31.50,40.00){\line(1,0){17.00}}
          \put(51.50,40.00){\line(1,0){17.00}}
          \put(0.00,60.00){\line(1,0){8.50}}
          \put(51.50,60.00){\line(1,0){17.00}}
          \put(0.00,80.00){\line(1,0){8.50}}
          \put(31.50,80.00){\line(1,0){17.00}}
          \put(71.50,80.00){\line(1,0){8.50}}
          \put(11.50,100.00){\line(1,0){17.00}}
          \put(31.50,100.00){\line(1,0){17.00}}
          \put(10.00,21.50){\line(0,1){17.00}}
          \put(10.00,41.50){\line(0,1){17.00}}
          \put(10.00,101.50){\line(0,1){8.50}}
          \put(30.00,10.00){\line(0,1){8.50}}
          \put(30.00,81.50){\line(0,1){17.00}}
          \put(50.00,21.50){\line(0,1){17.00}}
          \put(50.00,61.50){\line(0,1){17.00}}
          \put(70.00,10.00){\line(0,1){8.50}}
          \put(70.00,41.50){\line(0,1){17.00}}
          \put(70.00,81.50){\line(0,1){17.00}}
          \put(70.00,101.50){\line(0,1){8.50}}
           
          \multiput(10,20)(0,20){5}{\circle{3.00}}
          \multiput(30,20)(0,20){5}{\circle{3.00}}
          \multiput(50,20)(0,20){5}{\circle{3.00}}
          \multiput(70,20)(0,20){5}{\circle{3.00}}
          
          \put(23,62){$u$} \put(53,62){$v$}

          \put(120.00,20.00){\line(1,0){8.50}}
          \put(131.50,20.00){\line(1,0){17.00}}
          \put(171.50,20.00){\line(1,0){17.00}}
          \put(191.50,20.00){\line(1,0){8.50}}
          \put(151.50,40.00){\line(1,0){17.00}}
          \put(171.50,40.00){\line(1,0){17.00}}
          \put(120.00,60.00){\line(1,0){8.50}}
          \put(171.50,60.00){\line(1,0){17.00}}
          \put(120.00,80.00){\line(1,0){8.50}}
          \put(151.50,80.00){\line(1,0){17.00}}
          \put(191.50,80.00){\line(1,0){8.50}}
          \put(131.50,100.00){\line(1,0){17.00}}
          \put(151.50,100.00){\line(1,0){17.00}}
          \put(130.00,21.50){\line(0,1){17.00}}
          \put(130.00,41.50){\line(0,1){17.00}}
          \put(130.00,101.50){\line(0,1){8.50}}
          \put(150.00,10.00){\line(0,1){8.50}}
          \put(150.00,81.50){\line(0,1){17.00}}
          \put(170.00,21.50){\line(0,1){17.00}}
          \put(170.00,61.50){\line(0,1){17.00}}
          \put(190.00,10.00){\line(0,1){8.50}}
          \put(190.00,41.50){\line(0,1){17.00}}
          \put(190.00,81.50){\line(0,1){17.00}}
          \put(190.00,101.50){\line(0,1){8.50}}
          
          \multiput(130,20)(0,20){5}{\circle{3.00}}
          \multiput(150,20)(0,20){5}{\circle{3.00}}
          \multiput(170,20)(0,20){5}{\circle{3.00}}
          \multiput(190,20)(0,20){5}{\circle{3.00}}
          
          \put(143,62){$u$} \put(173,62){$v$}
                 
          {\color[rgb]{0,0,1} \put(151.50,60.00){\line(1,0){17.00}}
          \put(158,55){$e$}}
          \put(29.5,0){Copy 1} \put(149.5,0){Copy 2}
     \end{picture}
     
     \caption{Two appropriately coupled copies of the same i.i.d.\ percolation process
     on $\Z^d$ on which the opinion dynamics procedure takes place. \label{coupling}}
\end{figure}
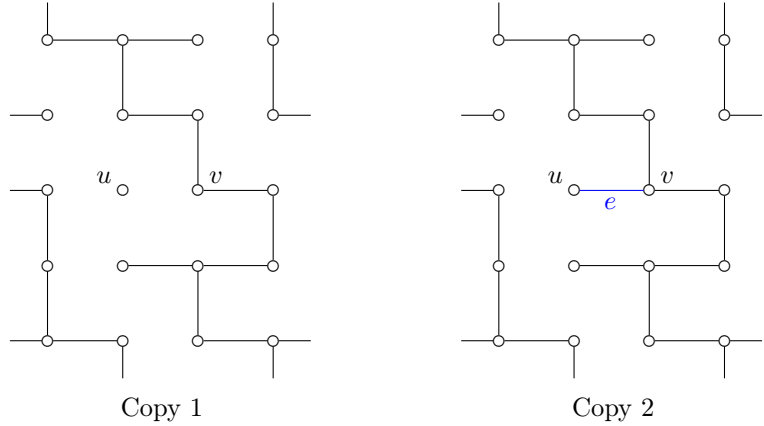\vspace*{-0.1cm}

\noindent
Now we want to couple the two copies in such a way that with positive probability $e$ is closed in copy 1
and open in copy 2 under the event $A$. Let $U$ be a $\text{unif}([0,1])$-distributed random variable, independent
of the percolation process on $E(\Z^d)\setminus\{e\}$. Declare $e$ to be open in copy 1 if $U<p$, closed
otherwise, and open in copy 2 if $U>1-p$ and closed otherwise. This defines two proper i.i.d. bond percolation
processes.

If $B$ denotes the event that $e$ is closed in copy 1 and open in copy 2, we get $\Prob(B)=\min\{p, 1-p\}>0$.
By independence we also have that the event $A\cap B$ has positive probability.

Since the event that there is strong consensus on the infinite percolation cluster is
ergodic with respect to shifts, it is a 0-1-event. Due to the assumption it must have
probability 1. Define $\delta:=\tfrac12-\theta$, which is positive.

Let us now restrict our attention to the event $A\cap B$ and the first copy. Since $v$ lies in the
infinite component, there is a time $T<\infty$ s.t. 
\begin{equation}\label{strongcon}
\Prob\big(|\eta_t(v)-\E\eta_0|<\tfrac{\delta}{2}\text{ for all }t\geq T\;|\;A\cap B)>0.
\end{equation}
Note that given $A\cap B$, in copy 1 the process $(\eta_t(v))_{t\geq0}$ is independent of $\eta_0(u)$
as well as the Poisson process attributed to $e$. By the choice of $\theta$ and the properties
of the initial distribution we get in addition:
\begin{equation*}
\Prob\big(\eta_0(u)\notin[\E\eta_0-(\theta+\tfrac{\delta}{2}),\E\eta_0+(\theta+\tfrac{\delta}{2})]\big)>0.
\end{equation*}
If we finally define $C$ to be the event that $A\cap B$ occurs, no Poission event occurs at
$e$ before $T$, $|\eta_t(v)-\E\eta_0|<\tfrac{\delta}{2}\text{ for all }t\geq T$ and $|\eta_0(u)-\E\eta_0|\geq\theta+\tfrac{\delta}{2}$, independence of the latter events conditioned
on $A\cap B$ makes sure that $C$ occurs with positive probability.\vspace{1em}

\noindent If we run the opinion dynamics on both copies simultaneously it is obvious that they
behave identically as long as no Poisson event occurs for $e$. Given the event $C$ the values
at $u$ and $v$ are further than $\theta$ apart from time $T$ on. Hence, even in the
second copy, there will never be an interaction between the two since no Poisson event occurs at $e$
before time $T$. In other words, with probability at least $\Prob(C)>0$ there will be no consensus
in the infinite percolation cluster of the second copy, to which given $A\cap B$ both $u$ and $v$ belong.
Since both copies underly the same distribution, this contradicts the assumption that we have strong
consensus. It is worth noting that strictly weak consensus can not be excluded since the argument in
(\ref{strongcon}) does not hold for the weak case.
\end{proof}

\begin{remark}
The two results of Theorem \ref{perc1} and \ref{perc2} put together imply the following:
The Deffuant model on the infinite cluster, formed by supercritical i.i.d.\ bond percolation on $\Z^d$ with
non-trivial percolation parameter $p\in(p_c,1)$, featuring i.i.d.\ initial opinions having a non-degenerate
marginal distribution on $[0,1]$ -- in the sense that it attributes positive probability to $(0,1),\ [0,\epsilon)$
and $(1-\epsilon,1]$ for all $\epsilon>0$ -- either approaches weak consensus for all $\theta\in(0,1]$ or
there is a phase transition in this parameter.
\end{remark}

\begin{remark}\label{weakercon}
Similarly to the ideas in Subsection \ref{dep}, we can relax the strong condition of independence
when it comes to the initial opinion values and still receive the same result. In the proof of Theorem \ref{perc2},
the only instance where more than stationarity and ergodicity with respect to shifts of the initial configuration
$\{\eta_0(v)\}_{v\in\Z^d}$ was used is in the conclusion that the event $C$ has positive probability.
This however can also be guaranteed without the independence of initial opinion values, if only
$\{\eta_0(v)\}_{v\in\Z^d}$ additionally satisfies the finite energy condition as laid down in Definition
\ref{finen} but now with $\Z^d$ in place of $\Z$.
\end{remark}

\subsection*{Acknowledgement}
We want to thank two referees for a very careful reading of and valuable comments to an earlier draft
that helped us not only to clarify and straighten out the statement and proof of some of the results
but also to further explore their scope. In particular, the extensions in Subsection \ref{dep} to dependent
initial configurations and in Remark \ref{amenable} to amenable graphs were triggered by their questions.


\vspace{0.5cm}
\makebox[\textwidth][c]{
\begin{minipage}[t]{1.2\textwidth}
\begingroup
	\begin{minipage}[t]{0.5\textwidth}
	{\sc \small Olle Häggström\\
   Department of Mathematical Sciences,\\
   Chalmers University of Technology,\\
   412 96 Gothenburg, Sweden.}\\
   olleh@chalmers.se
	\end{minipage}
	\hfill
	\begin{minipage}[t]{0.5\textwidth}
	{\sc \small Timo Hirscher\\
   Department of Mathematical Sciences,\\
   Chalmers University of Technology,\\
   412 96 Gothenburg, Sweden.}\\
   hirscher@chalmers.se\\
	\end{minipage}
	\endgroup
\end{minipage}}

\end{document}